\documentclass{article}

\usepackage[utf8]{inputenc}
\usepackage{amsmath, amssymb, amsthm,enumitem,bbm}
\usepackage{indentfirst}

\theoremstyle{definition}
\newtheorem{proposition}{Proposition}
\numberwithin{proposition}{section}
\newtheorem{lemma}{Lemma}
\numberwithin{lemma}{section}
\newtheorem{theorem}{Theorem}
\numberwithin{theorem}{section}

\usepackage{tocloft}

\usepackage[nottoc,notlot,notlof]{tocbibind}

\makeatletter
\renewcommand\@seccntformat[1]{\csname the#1\endcsname.\quad}
\newcommand{\vast}{\bBigg@{3}}
\newcommand{\Vast}{\bBigg@{4}}
\makeatother

\DeclareMathOperator{\charac}{char}

\DeclareFontFamily{U}{mathx}{\hyphenchar\font45}
\DeclareFontShape{U}{mathx}{m}{n}{<-> mathx10}{}
\DeclareSymbolFont{mathx}{U}{mathx}{m}{n}
\DeclareMathAccent{\widebar}{0}{mathx}{"73}

\usepackage{pict2e}

\newcounter{picno}
\newcommand{\pic}[1]{\refstepcounter{picno}\label{#1}}

\title{A Combinatorial Formula for Affine Hall-Littlewood Functions via a Weighted Brion Theorem}
\author{Boris Feigin and Igor Makhlin} \date{}

\begin{document}

\maketitle

\begin{abstract}
We present a new combinatorial formula for Hall-Littlewood functions associated with the affine root system of type $\tilde A_{n-1}$, i.e. corresponding to the affine Lie algebra $\widehat{\mathfrak{sl}}_n$. Our formula has the form of a sum over the elements of a basis constructed by Feigin, Jimbo, Loktev, Miwa and Mukhin in the corresponding irreducible representation. 

Our formula can be viewed as a weighted sum of exponentials of integer points in a certain infinite-dimensional convex polyhedron. We derive a weighted version of Brion's theorem and then apply it to our polyhedron to prove the formula.
\end{abstract}

\tableofcontents
\normalfont

\section{Introduction}

We start off by recalling the definition of Hall-Littlewood functions in the context of a general symmetrizable Kac-Moody algebra (see, for example,~\cite{carter}). Let $\mathfrak g$ be such an algebra with Cartan subalgebra $\mathfrak h$. Let $\Phi\subset \mathfrak h^*$ be its root system and $\Phi^+$ be the subset of positive roots, $\alpha\in\Phi$ having multiplicity $m_\alpha$. Finally, let $\lambda\in\mathfrak h^*$ be an integral dominant weight and $W$ be the Weyl group with length function $l$. The corresponding Hall-Littlewood function is then defined as
\begin{equation}\label{def}
P_\lambda=\frac 1{W_\lambda(t)}\sum\limits_{w\in W} w\left(e^\lambda \prod\limits_{\alpha\in\Phi^+}\left(\frac{1-t e^{-\alpha}}{1-e^{-\alpha}}\right)^{m_\alpha}\right).
\end{equation}
Here $W_\lambda(t)$ is the Poincaré series of the stabilizer $W_\lambda\subset W$, i.e. $$W_\lambda(t)=\sum\limits_{w\in W_\lambda}t^{l(w)}$$ (in particular, $W_\lambda(t)=1$ for regular $\lambda$).

Both sides of~(\ref{def}) should be viewed as elements of $\mathfrak R_t=\mathfrak R\otimes\mathbb{Z}[t]$, where $\mathfrak R$ is the ring of characters the support of which is contained in the union of a finite number of lower sets with respect to the standard ordering on $\mathfrak h^*$. It is easy to show that $P_\lambda$ is indeed a well-defined element of $\mathfrak R_t$ (see, for instance,~\cite{vis1}).

The definition~(\ref{def}) could be given only in terms of the corresponding root system eliminating any mention of Lie algebras and thus giving Hall-Littlewood functions a purely combinatorial flavor. The language of Kac-Moody algebras and their representations is, however, very natural when dealing with these objects.

It is worth noting that $P_\lambda$ specializes to the Kac-Weyl formula for the character of the irreducible representation $L_\lambda$ with highest weight $\lambda$ when $t=0$ and to $\sum_{w\in W}e^{w\lambda}$ when $t=1$. Thus it can be viewed as an interpolation between the two.

Another important observation is that once we've chosen a basis $\gamma_1,\ldots,\gamma_n$ in the lattice of integral weights, the $P_\lambda$ turn into formal Laurent series in corresponding variables $x_1,\ldots,x_n$ with coefficients in $\mathbb{Z}[t]$  (as does any other element of $\mathfrak R_t$). In the case of $\mathfrak{g}$ having finite type these Laurent series are, in fact, Laurent polynomials (the characters $P_\lambda$ have finite support) and are often referred to as ``Hall-Littlewood polynomials". We are, however, primarily interested in the affine case.

Our main result is a new combinatorial formula for the functions $P_\lambda$ in the case of $\mathfrak g=\mathfrak{\widehat{sl}_n}$ (root system of type $\tilde A_{n-1}$). One geometrical motivation for considering these expressions is as follows. Consider the group $\widehat{G}=\widetilde{SL}_n(\mathbb{C}[t,t^{-1}]),$ the central extension of the loop group of $SL_n(\mathbb{C})$ defined in the standard way. Next, consider the flag variety $F=\widehat{G}/B_+$, where $B_+$ is the Borel subgroup of $\widehat{G}$. On $F$ we have the sheaf of differentials $\Omega^*$ as well as the equivariant linear bundle $\mathcal L_\lambda$. It can be shown that the equivariant Euler characteristic of the sheaf $\Omega^*\otimes\mathcal L_\lambda$, namely $$\sum\limits_{i,j\ge 0}(-1)^i t^j\charac(H^i(F,\Omega^j\otimes\mathcal L_\lambda))$$ is precisely $W_\lambda(-t)P_\lambda(-t)$. In fact, in order to get rid of the factor $W_\lambda(-t)$ for singular $\lambda$, one may consider the corresponding parabolic flag variety with its sheaf of differential forms twisted by the corresponding equivariant linear bundle. In this context affine Hall-Littlewood functions appear, for example, in~\cite{gro}.

Another topic in which Hall-Littlewood functions of type $\tilde A$ appear is the representation theory of the double affine Hecke algebra (see~\cite{che}).

Our formula turns out to be similar in spirit to the combinatorial formula for classic Hall-Littlewood functions, that is of type $A$. The latter formula, found already in~\cite{mac}, is a sum over Gelfand-Tsetlin patterns, combinatorial objects enumerating a basis in $L_\lambda$. The formula we present is the sum over a basis in an irreducible integrable representation of the affine algebra $\widehat{\mathfrak{sl}}_n$, which was obtained in the works~\cite{fs,fjlmm1,fjlmm2}. Moreover, although at first the combinatorial set enumerating the latter basis seems to be very different from the set of Gelfand-Tsetlin patterns, a certain correspondence may be constructed which lets one then define the summands in the formula similarly to the classic case.

We consider it essential to review the finite case as well as the affine one in order to both illustrate the more complicated affine case and to emphasize the deep analogies between the two cases. For this last reason we will deliberately introduce certain conflicting notations, i.e. analogous objects in the finite and affine cases may be denoted by the same symbol. However, which case is being considered should always be clear from the context.

Our approach to proving the formula is based on Brion's theorem for convex polyhedra, originally due to~\cite{bri}. This formula expresses the sum of exponentials of integer points inside a rational polyhedron as a sum over the polyhedron's vertices. 

Let us first explain how the approach works in the classic case.

The set of Gelfand-Tsetlin patterns associated with weight $\lambda$ may be viewed as the integer points of the Gelfand-Tsetlin polytope. This means that the character of $L_\lambda$ is in fact a sum of certain exponentials of these integer points and may thus be computed via Brion's theorem. It turns that the contributions of most vertices are zero, while the remaining vertices provide the summands in the classic formula for Schur polynomial. This scenario is discussed in the paper~\cite{me1}.

Further, the mentioned combinatorial formula for Hall-Littlewood polynomials of type $A$ implies that, in this case, $P_\lambda$ is the sum of these same exponentials but this time with coefficients which are polynomials in $t$. We derive and employ a generalization of Brion's theorem which expresses weighted sums of exponentials of a certain type as, again, a sum over the vertices. Our weights turn out to be of this very type and we may thus apply our weighted version of Brion's theorem. Once again most vertices contribute zero, while the remaining contributions add up to give formula~(\ref{def}). 

Now it can be said that in the case of $\widehat{\mathfrak{sl}}_n$ the situation is similar. The set parametrizing the basis vectors can be, again, viewed as the set of integer points of a ``convex polyhedron"{}, this time, however, infinite dimensional. Moreover, the summand corresponding to each point is once more a certain exponential. One can prove a Brion-type formula for this infinite-dimensional polyhedron, which expresses the sum of exponentials as a sum over the vertices. Again, the contributions of most vertices are zero and the remaining contributions add up to the Kac-Weyl formula for $\charac{L_\lambda}$. This scenario is presented in~\cite{me2}. (To be accurate,~\cite{me2} deals with the Feigin-Stoyanovsky subspace and its character but the transition to the whole representation can be carried out rather simply, as shown in~\cite{fjlmm2}.)

Finally, our formula for affine Hall-Littlewood functions is, just like in the classic case, a sum of the same exponentials of integer points of the same polyhedron but with coefficients which are polynomials in $t$. We show that using our weighted version of Brion's theorem we may decompose this sum as a sum over the vertices. The same distinguished set of vertices will provide nonzero contributions which add up to formula~(\ref{def}), proving the result.

The below text is structured as follows. In Part~\ref{part1} we recall the preliminary results mentioned in the introduction and give the statement of our main result. Then we introduce our generalization of Brion's theorem and explain in more detail how it can be applied to proving our formula. In Part~\ref{tools} we develop the combinatorial arsenal needed to implement our proof. We introduce a family of polyhedra naturally generalizing Gelfand-Tsetlin polytopes and prove two key facts concerning those polyhedra. From the author's viewpoint, the topics discussed in Part~\ref{tools} are of some interest in their own right. In the last part we show how to obtain the weighted Brion-type formula for the infinite-dimensional polyhedron and then prove our central theorem concerning the contributions of vertices.

\part{Preliminaries, The Result and Idea of Proof}\label{part1}

\section{The Combinatorial Formula for Finite Type $A$}\label{gtcomb}

Let $g=\mathfrak{sl}_n$ and $\lambda\in\mathfrak h^*$ be an integral dominant nonzero weight. Let $$\lambda=(a_1,\ldots,a_{n-1})$$ with respect to a chosen basis of fundamental weights. In the appropriate basis $\lambda$ has coordinates $$\lambda_i=a_i+\ldots+a_{n-1}.$$ This will be our basis of choice in the lattice of integral weights, and we will view characters as Laurent polynomials in corresponding variables $x_1,\ldots,x_{n-1}$.

The Gelfand-Tsetlin basis in $L_\lambda$ is parametrized by the following objects known as Gelfand-Tsetlin patterns (abbreviated as GT-patterns.). Each such pattern is a number triangle $\{s_{i,j}\}$ with $0\le i\le n-1$ and $1\le j\le n-i$. The top row is given by $s_{0,j}=\lambda_j$ (with $s_{0,n}=0$) while the other elements are arbitrary integers satisfying the inequalities
\begin{equation}\label{gt}
s_{i,j}\ge s_{i+1,j}\ge s_{i,j+1}.
\end{equation}
The standard way to visualize these patterns is the following:
\begin{center}
\begin{tabular}{ccccccc}
$s_{0,1}$ &&$ s_{0,2}$&& $\ldots$ && $s_{0,n}$\\
&$s_{1,1}$ &&$ \ldots$&& $s_{1,n-1}$ &\\
&&$\ldots$ &&$ \ldots$& &\\
&&&$s_{n,1}$ &&&
\end{tabular}
\end{center}
Thus each number is no greater than the one immediately to its upper-left and no less than the one immediately to its upper-right, except for the numbers in row $0$ (row $i$ is comprised of the numbers $s_{i,*}$).

Let us denote the set of GT-patterns $\mathbf{GT}_\lambda$. For $A\in \mathbf{GT}_\lambda$ let $v_A$ be the corresponding basis vector, $v_A$ is a weight vector with weight $\mu_A$. If $A=(s_{i,j})$, then in the chosen basis $\mu_A$ has coordinates $$(\mu_A)_i=\sum\limits_j s_{i,j}-\sum\limits_j s_{i-1,j}.$$

Each $A\in \mathbf{GT}_\lambda$ also determines a polynomial in $t$ denoted $p_A$. We have
\begin{equation}\label{hlcoeff}
p_A=\prod\limits_{l=1}^{n-1} (1-t^l)^{d_l}
\end{equation}
where $d_l$ is the following statistic. It is the number of pairs consisting of $1\le i\le n-1$ and $a\in\mathbb{Z}$ such that the integer $a$ occurs $l$ times in row $i$ of $A$ and $l-1$ times in row $i-1$. The combinatorial formula is then as follows.
\begin{theorem}\label{combfin}
$$P_\lambda=\sum\limits_{A\in \mathbf{GT}_\lambda} p_A e^{\mu_A}.$$
\end{theorem}

This theorem is a direct consequence of the branching rule for classic Hall-Littlewood polynomials, which can be found in~\cite{mac}. One needs to note, however, that the definition in~\cite{mac} corresponds to the case of $\mathfrak{gl}_n$ rather than $\mathfrak{sl}_n$. Fortunately, this adjustment is fairly simple to make: the polynomial we have obtained is, in the notations of~\cite{mac}, just $$P_{(\lambda_1,\ldots,\lambda_{n-1},0)}(x_1,\ldots,x_{n-1},1;t).$$

\section{The Monomial Basis}\label{monbasis}

Just like the combinatorial formula for type $A$ is a sum over the elements of a basis in the irreducible $\mathfrak{sl}_n$-module, our formula for type $\tilde{A}$ is a sum over the elements of a certain basis in the irreducible $\widehat{\mathfrak{sl}}_n$-module. This basis was constructed by Feigin, Jimbo, Loktev, Miwa and Mukhin in the papers~\cite{fjlmm1,fjlmm2}, in this section we give a concise review of its properties.

Let $\lambda$ be an integral dominant $\widehat{\mathfrak{sl}}_n$-weight with coordinates $$(a_0,\ldots,a_{n-1})$$ with respect to a chosen basis of fundamental weights. The level of $\lambda$ is $k=\sum a_i$. 

The basis in $L_\lambda$ is parametrized by the elements of the following set $\mathbf{\Pi}_\lambda$. Each element $A$ of $\mathbf{\Pi}_\lambda$ is a sequence of integers $(A_i)$ infinite in both directions which satisfies the following three conditions.
\begin{enumerate}[label=\roman*)]
\item For $i\gg 0$ we have $A_i=0$.
\item For $i\ll 0$ we have $A_{i}=a_{i\bmod n}$.
\item For all $i$ we have $A_i\ge 0$ and $A_{i-n+1}+A_{i-n+2}+\ldots+A_i\le k$ (sum of any $n$ consecutive terms).
\end{enumerate}

The basis vector corresponding to $A\in\mathbf{\Pi}_\lambda$ is a weight vector with weight $\mu_A$. We will need an explicit description of $\mu_A$. First, observe that, since $\mu_A$ is in the support of $\charac(L_\lambda)$, the weight $\mu_A-\lambda$ is in the root lattice. In other words, we may fix a basis in the root lattice and describe $\mu_A-\lambda$ with respect to this basis. If $\alpha_0,\ldots,\alpha_{n-1}$ are the simple roots and $\delta$ is the imaginary root, then the basis consists of the roots $$\gamma_i=\alpha_1+\ldots+\alpha_i$$ for $1\le i\le n-1$ and the root $-\delta$.

Now consider $T^0\in\mathbf{\Pi}_\lambda$ given by $T^0_i=0$ when $i>0$ and $A_{i}=a_{(i\bmod n)}$ when $i\le 0$. The coordinates of $\mu_A-\lambda$ are determined by the termwise difference $A-T^0$ in the following way. The coordinate corresponding to $\gamma_i$ is equal to 
\begin{equation}\label{zweight}
\sum\limits_{q\in\mathbb{Z}}(A_{q(n-1)+i}-T^0_{q(n-1)+i}),
\end{equation}
while the coordinate corresponding to $-\delta$ is
\begin{equation}\label{qweight}
\sum\limits_{i\in\mathbb{Z}}\left\lceil\frac{i}{n-1}\right\rceil(A_i-T^0_i).
\end{equation}
For example one may now check that $\mu_{T^0}=\lambda$, i.e. $v_{T^0}$ is the highest weight vector.

We will refrain from giving an explicit definition of the vectors $v_A$ themselves, only pointing out that the basis is monomial. That means that every $v_A$ is obtained from the highest weight vector by the action of a monomial in the root spaces of the algebra $\widehat{\mathfrak{sl}}_n$. Thus this basis is of completely different nature than the Gelfand-Tsetlin basis, which makes the deep similarities between the affine and finite cases presented below, in a way, surprising.

\section{The Main Result}

One of the keys to our main result is the transition from infinite sequences comprising $\mathbf{\Pi}_\lambda$ to Gelfand-Tsetlin patterns (of sorts), which we mentioned in the introduction. It should be noted that analogues of GT patterns for the (quantum) affine scenario have been considered in papers~\cite{jimbo, ffnl, tsym}, our construction, however, differs substantially.

The object we associate with every $A\in\mathbf{\Pi}_\lambda$ is an infinite set of numbers $s_{i,j}(A)$ with both $i$ and $j$ arbitrary integers, for any $i,j$ satisfying the inequalities~(\ref{gt}). In general, we will refer to arrays of real numbers $(s_{i,j})$ satisfying~(\ref{gt}) as ``plane-filling GT-patterns". Similarly to classic GT-patterns, we visualize them as follows.
\begin{center}
\begin{tabular}{ccccccccc}
&$\ldots$ &&$\ldots$ && $\ldots$ && $\ldots$&\\
$\ldots$ &&$ s_{-1,-1}$&& $s_{-1,0}$ && $s_{-1,1}$&&\ldots\\
&$\ldots$ &&$ s_{0,-1}$&& $s_{0,0}$ &&\ldots&\\
$\ldots$ &&$ s_{1,-2}$&& $s_{1,-1}$ && $s_{1,0}$&&\ldots\\
&$\ldots$ &&$\ldots$ && $\ldots$ && $\ldots$&
\end{tabular}
\end{center}

To generalize the definition of $T^0$, for any $m\in\mathbb{Z}$ let $T^m$ be given by $T^m_i=0$ when $i>mn$ and $A_{i}=a_{(i\bmod n)}$ when $i\le mn$. Then, by definition,
\begin{equation}\label{gtdef}
s_{i,j}(A)=\sum\limits_{l\le in+j(n-1)}(A_l-T_l^{i+j})-\sum_{l=in+j(n-1)+1}^{(i+j)n}T^{i+j}_l.
\end{equation}
Note that the first sum on the right has a finite number of nonzero summands and the second sum is nonzero only when $j>0$. A way to rephrase this definition is to say that we take the sum of all terms of the sequence obtained from $A$ by setting all terms with number greater than $in+j(n-1)$ to zero and then subtracting $T^{i+j}$.

We will also use definition~(\ref{gtdef}) in the more general context of $A$ being any sequence satisfying i) and ii) from the definition in the previous section but not necessarily iii).

\begin{proposition}
If $A\in\mathbf{\Pi}_\lambda$, then the array $(s_{i,j}(A))$ constitutes an plane-filling GT-pattern.
\end{proposition}
\begin{proof}
One observes that $$s_{i,j}(A)-s_{i-1,j+1}(A)=A_{in+j(n-1)}\ge 0$$ and 
\begin{multline*}
s_{i,j}(A)-s_{i-1,j}(A)=A_{(i-1)n+j(n-1)+1}+\ldots+A_{in+j(n-1)}+\\-T^{i+j}_{(i+j-1)n+1}-\ldots-T^{i+j}_{(i+j)n}=A_{(i-1)n+j(n-1)+1}+\ldots+A_{in+j(n-1)}-k\le 0.
\end{multline*}
In other words, the numbers $s_{i,j}(A)$ satisfy the inequalities~(\ref{gt}).  
\end{proof}
The proof shows that $s_{i,j}(A)=s_{i-1,j+1}(A)$ if and only if $A_{in+j(n-1)}=0$ and $s_{i,j}(A)=s_{i-1,j}(A)$ if and only if $A_{in+(j-1)(n-1)}+\ldots+A_{in+j(n-1)}=k$ (sum of $n$ consecutive terms). This observation should be kept in mind when dealing with plane-filling GT-patterns $(s_{i,j}(A))$.

Now to give the statement of our main theorem we associate with every sequence $A$ satisfying i) and ii) a weight $p(A)$ of form $\prod_{l=1}^n (1-t^l)^{d_l}$. The integers $d_l$ are defined in terms of the associated plane-filling GT-pattern. Once again, to define $d_l$ we consider the set of pairs of integers $(x,i)$ such that the number $x$ appears $l-1$ times in row $i-1$ and $l$ times in row $i$ of $(s_{i,j}(A))$. The set of such pairs is, however, likely to be infinite and $d_l$ is, in fact, the size of a factor set with respect to a certain equivalence relation which we now describe.

One of the key features of the array $(s_{i,j})=(s_{i,j}(A))$ is the easily verified equality
\begin{equation}\label{shift}
s_{i-n+1,j+n}=s_{i,j}-k
\end{equation}
holding for any $i,j$. Now consider the set $X_l$ of pairs $(i,j)$ for which $$s_{i-1,j}\neq s_{i-1,j+1}=\ldots=s_{i-1,j+l-1}\neq s_{i-1,j+l}$$ and $$s_{i,j-1}\neq s_{i,j}=\ldots=s_{i,j+l-1}\neq s_{i,j+l}.$$ $X_l$ is in an obvious bijection with the set of pairs $(x,i)$ defined above. Equality~(\ref{shift}) shows that if $(i,j)\in X_l$, then $(i-\alpha(n-1),j+\alpha n)\in X_l$ for any integer $\alpha$. Our relation is defined by $(i,j)\sim (i-n+1,j+n)$.
\begin{proposition}\label{finite}
The set $X_l/\sim$ is finite.
\end{proposition}
\begin{proof}
First, every equivalence class in $X_l$ contains exactly one representative $(i,j)$ with $1\le i\le n-1$. Therefore, it suffices to show that the number of $(i,j)\in X_l$ with $i$ within these bounds is finite. Further, the following two facts are straightforward from definition~(\ref{gtdef}) and $A$ satisfying i) and ii).
\begin{enumerate}[label=\arabic*)]
\item If $1\le i\le n-1$, then for $j\gg 0$ one has $s_{i,j+1}=s_{i,j}-k$.
\item If $1\le i\le n-1$, then for $j\ll 0$ one has $s_{i,j+1}=s_{i,j}$ if and only if $a_{j\bmod n}=0$ and thus if and only if $s_{i-1,j+1}=s_{i-1,j}$ holds as well.
\end{enumerate}
However, 1) shows that if $(i,j)\in X_l$ and $i\in[1,n-1]$, then $j$ can not be arbitrarily large, while 2) shows that $-j$ can not be arbitrarily large.
\end{proof}

We can now define $d_l=|X_l/\sim|$ and state our main result.
\begin{theorem}\label{main}
For an integral dominant nonzero $\widehat{\mathfrak{sl}}_n$-weight $\lambda$ one has 
\begin{equation}\label{mainformula}
P_\lambda=\sum\limits_{A\in\mathbf{\Pi}_\lambda}p(A) e^{\mu_A}.
\end{equation}
\end{theorem}

{\bf Remark.} As one can see, in the case of $\lambda=0$ the set $\mathbf\Pi_\lambda$ consists of a single zero sequence. The corresponding plane-filling GT pattern is also identically zero and our definition of $p(A)$ falls apart. In a sense, the case of $\lambda=0$ being exceptional is caused by the fact that for an affine root system the stabilizer of $0$ is infinite, unlike any other integral weight. This ultimately leads to definition~(\ref{def}) rendering $P_0$ not equal to $1$, unlike any root system of finite type.

\section{Brion's Theorem and its Generalization}

In this section we give a concise introduction to Brion's theorem and then present our generalization. After that we will elaborate on the connection between these subjects and our formula.

Consider a vector space $\mathbb{R}^m$ with a fixed basis and corresponding lattice of integer points $\mathbb{Z}^m\subset\mathbb{R}^m$.  For any set $P\subset\mathbb{R}^m$ one may consider its characteristic series $$S(P)=\sum\limits_{a\in P\cap\mathbb{Z}^m} e^a,$$ a formal Laurent series in the variables $x_1,\ldots,x_m$. (Once we have assigned a formal variable to each basis vector we may define the monomial $e^a=x_1^{a_1}\ldots x_m^{a_m}$.)

If $P$ is a rational convex polyhedron (a set defined by a finite number of non-strict linear inequalities with integer coefficients, not necessarily bounded) it can be shown that there exists a Laurent polynomial $q\in\mathbb{Z}[x_1^{\pm 1},\ldots,x_m^{\pm 1}]$ such that $qS(P)$ is also some Laurent polynomial. Moreover, the rational function $\tfrac{qS(P)}q$ does not depend on the choice of $q$ and is denoted $\sigma(P)$. This function is known as the integer point transform (IPT) of $P$.

For a vertex $v$ of $P$ let $C_v$ be the tangent cone to $P$ at $v$. Brion's theorem is then the following identity.
\begin{theorem}[\cite{bri,kp}]\label{brion}
In the field of rational functions we have $$\sigma(P)=\sum\limits_{v\text{ vertex of }P}\sigma(C_v).$$
\end{theorem}

A nice presentation of these topics can be found in the books~\cite{barv,beckrob}.

Our generalization of Theorem~\ref{brion} is stated in the following setting. Suppose we have a convex rational polyhedron $P\subset\mathbb{R}^n$. Let $R$ be an arbitrary commutative ring, and consider any map $\varphi:\mathcal{F}_P\rightarrow R$, where we use $\mathcal F_P$ to denote the set of faces of $P$. The map $\varphi$ defines a function $g:P\rightarrow R$, where for $x\in P$ we have $g(x)=\varphi(f)$ with $f$ being the face of minimal dimension containing $x$.

Next, consider the weighted generating function $$S_\varphi(P)=\sum\limits_{a\in P\cap\mathbb{Z}^n}g(a)\exp(a)\in R[[x_1^{\pm1},\ldots,x_n^{\pm 1}]].$$
\begin{proposition}
There exists a polynomial $Q\in R[x_1,\ldots,x_n]$ such that $QS_\varphi(P)\in R[x_1^{\pm1},\ldots,x_n^{\pm1}]$.
\end{proposition}
\begin{proof}
This follows from the fact that we may present a finite set of nonintersecting subpolyhedra of $P$, the union of which contains any lattice point in $P$ and on each of which $g$ is constant.

Namely, for each face we may consider its image under a dilation centered at its interior point with rational coefficient $0<\alpha<1$ large enough for the image to contain all of the face's interior lattice points.  $\{P_i\}$, the union of the obtained set of polyhedra with the set of all of $P$'s vertices has the desired properties. Thus $Q$ may be taken equal to the product of all the denominators of the rational functions $\sigma(P_i)$.

An important observation is that we may, therefore, actually take $Q$ to equal the product of of $1-e^\varepsilon$ over all minimal integer direction vectors $\varepsilon$ of edges of $P$, just like in the unweighted case.
\end{proof}

Thus we obtain a (well-defined) weighted integer point transform $$\sigma_\varphi(P)=\frac{QS_v(P)}Q\in R(x_1,\ldots,x_n).$$

Now note that if $v$ is a vertex of $P$ with tangent cone $C_v$, then there is a natural embedding $\mathcal F_{C_v}\hookrightarrow F_{P}$. If we allow ourselves to also use $\varphi$ to denote the restriction of $\varphi$ to $\mathcal F_{C_v}$ then our weighted Brion theorem can be stated as follows.

\begin{theorem}\label{wbrion}
$\sigma_\varphi(P)=\sum\limits_{v \text{ vertex of } P} \sigma_\varphi(C_v).$
\end{theorem}
\begin{proof}
Consider once again the set $\{P_i\}$ of polyhedra from the proof of the proposition. These polyhedra are in one-to-one correspondence with with $P$'s faces. Evidently, if we write down the regular Brion theorem for each of these polyhedra and then add these identities up with coefficients equal to the values of $\varphi$ at the corresponding faces, we end up with precisely the statement of our theorem.
\end{proof}

With the necessary adjustments, $R$ could actually be any abelian group. We, however, are interested in the specific case of $R=\mathbb{Z}[t]$.

\section{Employing the Weighted Brion Theorem in the Finite Case}\label{gtbrion}

First, we explain how this works in the classic case.

Consider the ${n+1}\choose 2$-dimensional real space with its coordinates labeled by pairs of integers $(i,j)$ such that $i\in[0,n-1]$ and $j\in[1,n-i]$. We then may view the elements of $\mathbf{GT}_\lambda$ as the integer points of the Gelfand-Tsetlin polytope, which we  denote $GT_\lambda$. This polytope consists of points with coordinates $s_{i,j}$ satisfying~(\ref{gt}). (Visibly, $GT_\lambda$ is contained in a $n\choose2$-dimensional affine subspace obtained by fixing the coordinates in the row 0.)

With each GT-pattern we now associate two Laurent monomials. One is $e^{\mu_A}$, a monomial in ${x_1,\ldots,x_{n-1}}$ as explained in Section~\ref{gtcomb}. The other one is $e^A$, a monomial in ${n+1} \choose 2$ variables, the exponential of a point in $\mathbb{R}^{{n+1} \choose 2}$. We denote these variables $\{t_{i,j}\}$. 

Now it is easily seen that $e^{\mu_A}$ is obtained from $e^A$ by the specialization 
\begin{equation}\label{specfin}
t_{i,j}\longrightarrow x_i^{-1} x_{i+1}
\end{equation}
(within this section we set $x_0=x_n=1$).
In general, for a rational function $$Q\in\mathbb{Z}[t](\{t_{i,j}\})$$ we denote the result of applying~(\ref{specfin}) to $Q$ by $F(Q)$ which, when well-defined, is an element of $\mathbb{Z}[t](x_1,\ldots,x_{n-1})$.

To make use of Theorem~\ref{wbrion} we need one more simple observation. For a GT-pattern $A$, the weight $p_A$ depends only on which of the inequalities~(\ref{gt}) are actually equalities for this specific pattern. These inequalities, however, define our polytope and therefore $p_A$ only depends on the minimal face of $GT_\lambda$ containing $A$. Therefore we have a weight function $$\varphi:\mathcal F_{GT_\lambda}\rightarrow\mathbb{Z}[t]$$ as discussed in the previous section.

We now see that the right-hand side in Theorem~\ref{combfin} can be expressed by applying our weighted Brion theorem to $GT_\lambda$ and $\varphi$ and then applying specialization $F$. The result of this procedure is described by the following theorem, which visibly implies Theorem~\ref{combfin}.
\begin{theorem}\label{contribfin}
There's a distinguished subset of vertices of $GT_\lambda$ parametrized by elements of the orbit $W\lambda$. For vertex $v$ corresponding to some $\mu\in W\lambda$ we have $$F(\sigma_{\varphi}(C_v))=\sum\limits_{w\lambda=\mu} w\left(e^\lambda \prod\limits_{\alpha\in\Phi^+}\left(\frac{1-t e^{-\alpha}}{1-e^{-\alpha}}\right)^{m_\alpha}\right).$$ For any $v$ outside this distinguished subset we have $F(\sigma_{\varphi}(C_v))=0$.
\end{theorem}

Interestingly enough, for a regular weight $\lambda$ this distinguished subset of vertices is precisely the set of simple vertices. As mentioned in the introduction, how and why this works out in the case of $t=0$ is shown in the preprint~\cite{me1}.

Since Theorem~\ref{combfin} itself is a well known result, we will not give a detailed proof of Theorem~\ref{contribfin}. However, it is rather easily deduced from the statements we do prove as will be briefly explained in the end Part~\ref{tools}.

\section{Employing the Weighted Brion Theorem in the Affine Case}\label{affbrion}

We now move on to the main affine case which is ideologically very similar but, of course, infinite-dimensional and thus technically more complicated.

Consider the real countable dimensional space $\Omega$ of sequences $x$ infinite in both directions for which one has $x_i=0$ when $i\gg 0$ and $x_{i}=x_{i-n}$ when $i\ll 0$ (for $x\in\Omega$ we denote $x_i$ the terms of this sequence). We denote the lattice of integer sequences $\mathbb{Z}^{{}^\infty}\subset\Omega$. In $\Omega$ we also have the affine subspace $V$ of sequences $x$ for which $x_i=a_{i\bmod n}$ when $i\ll 0$. Note that the functions $s_{i,j}(x)$ and $p(x)$ are defined precisely for $x\in V$.

Define the functionals $\chi_i$ on $\Omega$ taking $x$ to $x_{i-n+1}+\ldots+x_i$.  In these notations, the set $\mathbf{\Pi}_\lambda$ is precisely $\Pi\cap\mathbb{Z}^{{}^\infty}$, where $\Pi\subset V$ is the ``polyhedron"{} defined by the inequalities $x_i\ge 0$ and $\chi_i(x)\le k$ for all $i$.

It will often be more convenient to consider the translated polyhedron $$\bar{\Pi}=\Pi-T^0.$$ Geometrically and combinatorially the two polyhedra are identical, the advantage of $\bar{\Pi}$ is that it is contained in the linear subspace $\bar V\subset\Omega$ of sequences with a finite number of nonzero terms. For compactness use the $\ \bar{}\ $ notation to denote the $-T^0$ translation in general in the following two ways. If $X$ is a point or subset in $V$ we denote $\bar X=X-T^0$. If $\Phi$ is a map the domain of which consists of points or subsets in $V$, we define $\bar\Phi(\bar X)=\Phi(X)$.

To any integer sequence $x\in \bar V$ we may assign its formal exponent $e^x$, a (finite!) monomial in the infinite set of variables $\{t_i,i\in\mathbb{Z}\}$. Also, for $A\in\mathbf{\Pi}_\lambda$ the weight $\mu_A-\lambda$ is an integral linear combination of $\gamma_1,\ldots,\gamma_{n-1},-\delta$. Consequently, we may view $e^{\mu_A-\lambda}$ as a monomial in the corresponding variables $z_1,\ldots,z_{n-1},q$. Formulas~(\ref{zweight}) and~(\ref{qweight}) show that $e^{\mu_A-\lambda}$ is obtained from $e^{\bar A}$ by the specialization 
\begin{equation}\label{affspec}
t_i\longrightarrow z_{i\bmod (n-1)}q^{\left\lceil\tfrac{i}{n-1}\right\rceil},
\end{equation}
where the remainder is taken from $[1,n-1]$. In general we will denote the above specialization $G$, it being applicable to (some) expressions in the $t_i$.

Now we present a (weighted) Brion-type formula for $\bar\Pi$. One may define the faces of $\Pi$ and $\bar\Pi$ in a natural way (which will be done below). Of course, $f\subset\Pi$ is a face if and only if $\bar f\subset\bar\Pi$ is a face. One will see that $p(x)$ depends only on the minimal face of $\Pi$ containing $x$. In other words there is a map $$\varphi:\mathcal F_{\Pi}\rightarrow\mathbb{Z}[t]$$ such that $p(x)=\varphi(f)$ for the minimal face $f$ containing $x$. Denote $$S_{\bar\varphi}(\bar\Pi)=\sum_{x\in\bar\Pi\cap\mathbb{Z}^{{}^\infty}}\bar p(x)e^x.$$

Our formula will be an identity in the ring $\mathfrak S$ of those Laurent series in $q$ with coefficients in the field $\mathbb{Z}[t](z_1,\ldots,z_{n-1})$ which contain only a finite number of negative powers of $q$. This ring is convenient for the following reason. Consider a sequence of monomials $y_1,y_2,\ldots$ in variables $z_1,\ldots,z_{n-1},q$. If only a finite number of the $y_i$ contain a non-positive power of $q$ and none of them are equal to 1, then the product 
\begin{equation}\label{invprod}
(1-y_1)(1-y_2)\ldots
\end{equation}
is a well-defined element of $\mathfrak S$ and, most importantly, is invertible therein.

With each vertex $\bar v$ of $\bar\Pi$ we will associate a series $\tau_{\bar v}\in\mathfrak S$. This series will, in a certain sense, be the result of applying $G$ to an ``integer point transform" of the tangent cone $C_{\bar v}$ (also defined below). Our formula will then simply read as follows.
\begin{theorem}\label{infbrion}
In $\mathfrak S$ one has the identity $$G(S_{\bar\varphi}(\bar\Pi))=\sum_{\substack{\bar v\text{ vertex}\\\text{of }\bar\Pi}}\tau_{\bar v}.$$
\end{theorem}

Now, Theorem~\ref{main} may be rewritten as
\begin{equation}\label{rewrmain}
P_\lambda=e^{\lambda}G(S_{\bar\varphi}(\bar\Pi)).
\end{equation}
In view of this, Theorem~\ref{main} now follows from the following statement which is the affine analogue of Theorem~\ref{contribfin}.

\begin{theorem}\label{contrib}
There's a distinguished subset of vertices of $\Pi$ parametrized by elements of the orbit $W\lambda$ with the following two properties.
\begin{enumerate}[label=\alph*)]
\item For $v$ from this distinguished subset corresponding to $\mu\in W\lambda$ one has $$\tau_{\bar v}=\frac1{W_\lambda(t)}\sum\limits_{w\lambda=\mu}\frac{e^{w\lambda-\lambda} w\left(\prod\limits_{\alpha\in\Phi^+}\left(1-t e^{-\alpha}\right)^{m_\alpha}\right)}{w\left(\prod\limits_{\alpha\in\Phi^+}\left(1-e^{-\alpha}\right)^{m_\alpha}\right)},$$ where $v$ corresponds to weight $\mu\in W\lambda$.
\item For any other vertex $v$ of $\Pi_\lambda$ one has $\tau_{\bar v}=0$.
\end{enumerate}
\end{theorem}
The expression in the right-hand side in part a) is an element of $\mathfrak S$ because its denominator is a product of the type concerned in~(\ref{invprod}).

\part{Combinatorial Tools: Generalized Gelfand-Tsetlin Polyhedra}\label{tools}

In this Part we discuss certain finite-dimensional polyhedra which are seen to generalize Gelfand-Tsetlin polytopes.  The acquired tools will be applied to to the proof of Theorem~\ref{contrib} in the next part.

\section{Ordinary Subgraphs and Associated Polyhedra}

Consider an infinite square lattice as a graph $\mathcal R$ the vertices being the vertices of the lattice and the edges being the segments joining adjacent vertices. We visualize this lattice being rotated by $45^\circ$, i.e. the segments forming a $45^\circ$ angle with the horizontal.

We enumerate the vertices in accordance with our numbering of the elements of plane-filling GT-patterns. That is the vertices are enumerated by pairs of integers $(i,j)$. The set of vertices $(i,\cdot)$ form a row, they are the set of vertices situated on the same horizontal line. Within a row the second index increases from left to right and the two vertices directly above $(i,j)$ are $(i-1,j)$ and $(i-1,j+1)$.

We term a subgraph $\Gamma$ of $\mathcal R$ ``ordinary"{} if it has the following properties.
\begin{enumerate}
\item $\Gamma$ is a finite connected full subgraph.
\item  Whenever both $(i,j)\in\Gamma$ (short for $(i,j)$ is a vertex of $\Gamma$) and $(i,j+1)\in\Gamma$ we also have $(i+1,j)\in\Gamma$.
\item Let $a_\Gamma$ be the number of the top row containing vertices of $\Gamma$. If $i>a_\Gamma$, then whenever both $(i,j)\in\Gamma$ and $(i,j+1)\in\Gamma$ we also have $(i-1,j+1)\in\Gamma$.
\end{enumerate}
Note that $(i-1,j+1)$ and $(i+1,j)$ are the two common neighbors of $(i,j)$ and $(i,j+1)$. Below are some examples of what such a subgraph may look like.

\setlength{\unitlength}{1mm}
\begin{picture}(95,50)
\put(5,45){\line(1,-1){5}} \put(15,45){\line(-1,-1){5}} \put(15,45){\line(1,-1){5}} \put(25,45){\line(-1,-1){5}}
\put(10,40){\line(1,-1){5}} \put(20,40){\line(-1,-1){5}} \put(20,40){\line(1,-1){5}}
\put(15,35){\line(-1,-1){5}} \put(15,35){\line(1,-1){5}} \put(25,35){\line(-1,-1){5}} \put(25,35){\line(1,-1){5}}
\put(10,30){\line(1,-1){5}} \put(20,30){\line(-1,-1){5}} \put(20,30){\line(1,-1){5}} \put(30,30){\line(-1,-1){5}}
\put(15,25){\line(1,-1){5}} \put(25,25){\line(-1,-1){5}}
\put(20,20){\line(-1,-1){5}}
\pic{ex1}
\put(15,5){Figure \ref{ex1}}

\put(55,40){\line(-1,-1){5}} \put(55,40){\line(1,-1){5}}
\put(50,35){\line(-1,-1){5}} \put(50,35){\line(1,-1){5}} \put(60,35){\line(-1,-1){5}}
\put(45,30){\line(1,-1){5}} \put(55,30){\line(-1,-1){5}}
\put(50,25){\line(1,-1){5}}
\pic{ex2}
\put(45,5){Figure \ref{ex2}}

\put(80,45){\line(1,-1){5}} \put(90,45){\line(-1,-1){5}}
\put(85,40){\line(-1,-1){5}}
\put(80,35){\line(-1,-1){5}} \put(80,35){\line(1,-1){5}}
\put(75,30){\line(-1,-1){5}} \put(75,30){\line(1,-1){5}} \put(85,30){\line(-1,-1){5}} \put(85,30){\line(1,-1){5}}
\put(70,25){\line(1,-1){5}} \put(80,25){\line(-1,-1){5}} \put(80,25){\line(1,-1){5}} \put(90,25){\line(-1,-1){5}}
\put(75,20){\line(1,-1){5}} \put(85,20){\line(-1,-1){5}}
\pic{ex3}
\put(75,5){Figure \ref{ex3}}
\end{picture}

Note that every ordinary graph has one vertex in its last nonempty row.

Suppose $\Gamma$ has $l_\Gamma$ vertices in its top row. With each $\Gamma$ and nonincreasing sequence of integers $b_1,\ldots,b_{l_\Gamma}$ we associate a finite-dimensional rational polyhedron $D_\Gamma(b_1,\ldots,b_{l_\Gamma})$ in the countable-dimensional real space with coordinates enumerated by the vertices of $\mathcal R$. Consider a point $s$ in this space with its $(i,j)$-coordinate equal to $s_{i,j}$. By definition, $s\in D_\Gamma(b_1,\ldots,b_{l_\Gamma})$ if it satisfies the following requirements.
\begin{enumerate}
\item If $(i,j)\not\in\Gamma$, then $s_{i,j}=0$.
\item The $l_\Gamma$ coordinates in row $a_\Gamma$ are equal to $b_1,\ldots,b_{l_\Gamma}$ in that order from left to right.
\item\label{facets} For any $(i,j)\in\Gamma$ we have $s_{i-1,j}\ge s_{i,j}$ whenever $(i-1,j)\in\Gamma$ and $s_{i,j}\ge s_{i-1,j+1}$ whenever $(i-1,j+1)\in\Gamma$. In other words, for any two adjacent vertices of $\Gamma$ the corresponding inequality of type~(\ref{gt}) holds.
\end{enumerate}
Such polyhedra are a natural generalization of Gelfand-Tsetlin polytopes, the latter being $D_{\mathcal T}(b_1,\ldots,b_n)$, where $\mathcal T\subset\mathcal R$ is the ordinary subgraph with vertices $(i,j)$ for $0\le i\le n-1$ and $1\le j\le n-i$.

Any $s\in D_\Gamma(b_1,\ldots,b_{l_\Gamma})$ defines a subgraph of $\Gamma$ the vertices of which are the vertices of $\Gamma$ and edges are edges of $\Gamma$ for which the two corresponding coordinates in $s$ are equal. Since the polyhedron $D_\Gamma(b_1,\ldots,b_{l_\Gamma})$ is defined by the inequalities in correspondence with the edges of $\Gamma$, one sees that two points define the same subgraph if and only if the minimal faces containing them coincide. For this reason we have the following description of the faces of $D_\Gamma(b_1,\ldots,b_{l_\Gamma})$.

\begin{proposition}\label{faces}
The faces of $D_\Gamma(b_1,\ldots,b_{l_\Gamma})$ are in bijection with subgraphs of $\Gamma$ containing all vertices of $\Gamma$ and with the following properties.
\begin{enumerate}
\item Whenever two adjacent vertices of $\Gamma$ are in the same connected component of the subgraph they are also adjacent in the subgraph.
\item Whenever $(i,j)$ and $(i,j+1)$ are in the same component of the subgraph so are $(i+1,j)$ and $(i-1,j+1)$ (the latter when $i>a_\Gamma$).
\item The $i$-th and $j$-th vertex in row $a_\Gamma$ (counting from left to right) are in the same component of the subgraph if and only if $b_i=b_j$.
\end{enumerate}
The face corresponding to subgraph $\Delta$ consists of the points for which any two coordinates corresponding to adjacent vertices of $\Delta$ are equal. The dimension of the face is the number of those connected components in $\Delta$, which do not contain a vertex from row $a_\Gamma$.
\end{proposition}
\begin{proof}
If subgraph $\Delta$ has these properties it is straightforward to define a point $(s_{i,j})\in D_\Gamma(b_1,\ldots,b_{l_\Gamma})$ such that for two vertices $(i_1,j_1)$ and $(i_2,j_2)$ one has $s_{i_1,j_1}=s_{i_2,j_2}$ if and only if these vertices are in the same connected component of $\Delta$.

The statement concerning the dimension follows from the following observation. If $\Delta$ corresponds to face $f$, then for any point in $f$ all its coordinates in a component of $\Delta$ containing the $i$-th vertex from the top row are fixed and equal to $b_i$. Thus, when choosing a point in $f$, the degree of freedom is the number of components without a vertex from the top row.
\end{proof}
If $f$ is a face of some $D_\Gamma(b_1,\ldots,b_{l_\Gamma})$ we denote corresponding subgraph simply $\Delta_f$, the graph $\Gamma$ and values $b_1,\ldots,b_{l_\Gamma}$ being implicit. Note that, in particular, any connected component of $\Delta_f$ is itself an ordinary graph.

We now define a weight function $$\varphi_\Gamma(b_1,\ldots,b_{l_\Gamma}):\mathcal F_{D_\Gamma(b_1,\ldots,b_{l_\Gamma})}\rightarrow \mathbb{Z}[t].$$ The value of $\varphi_\Gamma(b_1,\ldots,b_{l_\Gamma})(f)$ is defined in terms of the graph $\Delta_f$. Namely, it is the product $\prod(1-t^l)^{d_l}$, where $d_l$ is the following statistic. It is the number of pairs $(E,i)$ where $E$ is a connected component of $\Delta_f$ and $i>a_\Gamma$ is an integer, such that $E$ contains exactly $l-1$ vertices from row $i-1$ and $l$ vertices from row $i$.

Here are three subgraphs of the three examples above accompanied by the dimension and weight of the corresponding face.

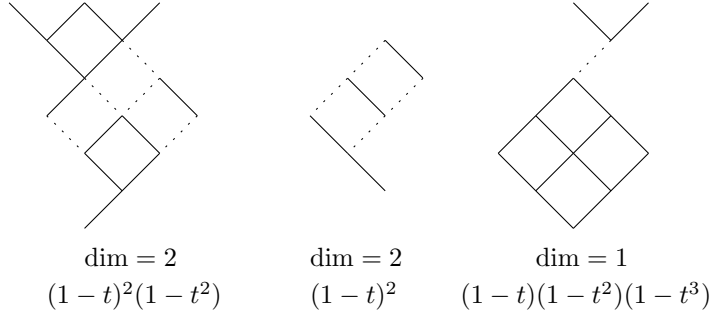
\begin{picture}(95,50)
\put(5,45){\line(1,-1){5}} \put(15,45){\line(-1,-1){5}} \put(15,45){\line(1,-1){5}} \put(25,45){\line(-1,-1){5}}
\put(10,40){\line(1,-1){5}} \put(20,40){\line(-1,-1){5}} \multiput(20,40)(1,-1){5}{\line(1,-1){0.2}}
\put(15,35){\line(-1,-1){5}} \multiput(15,35)(1,-1){5}{\line(1,-1){0.2}} \multiput(25,35)(-1,-1){5}{\line(-1,-1){0.2}} \put(25,35){\line(1,-1){5}}
\multiput(10,30)(1,-1){5}{\line(1,-1){0.2}} \put(20,30){\line(-1,-1){5}} \put(20,30){\line(1,-1){5}} \multiput(30,30)(-1,-1){5}{\line(-1,-1){0.2}}
\put(15,25){\line(1,-1){5}} \put(25,25){\line(-1,-1){5}}
\put(20,20){\line(-1,-1){5}}
\put(15,10){$\dim=2$}
\put(10,5){$(1-t)^2(1-t^2)$}

\multiput(55,40)(-1,-1){5}{\line(-1,-1){0.2}} \put(55,40){\line(1,-1){5}}
\multiput(50,35)(-1,-1){5}{\line(-1,-1){0.2}} \put(50,35){\line(1,-1){5}} \multiput(60,35)(-1,-1){5}{\line(-1,-1){0.2}}
\put(45,30){\line(1,-1){5}} \multiput(55,30)(-1,-1){5}{\line(-1,-1){0.2}}
\put(50,25){\line(1,-1){5}}
\put(45,10){$\dim=2$}
\put(45,5){$(1-t)^2$}

\put(80,45){\line(1,-1){5}} \put(90,45){\line(-1,-1){5}}
\multiput(85,40)(-1,-1){5}{\line(-1,-1){0.2}}
\put(80,35){\line(-1,-1){5}} \put(80,35){\line(1,-1){5}}
\put(75,30){\line(-1,-1){5}} \put(75,30){\line(1,-1){5}} \put(85,30){\line(-1,-1){5}} \put(85,30){\line(1,-1){5}}
\put(70,25){\line(1,-1){5}} \put(80,25){\line(-1,-1){5}} \put(80,25){\line(1,-1){5}} \put(90,25){\line(-1,-1){5}}
\put(75,20){\line(1,-1){5}} \put(85,20){\line(-1,-1){5}}
\put(75,10){$\dim=1$}
\put(65,5){$(1-t)(1-t^2)(1-t^3)$}
\end{picture}

For integers $b_1\ge\ldots\ge b_{l_\Gamma}$ the expression 
\begin{equation}\label{ipt}
\sigma_{\varphi_\Gamma(b_1,\ldots,b_{l_\Gamma})}(D_\Gamma(b_1,\ldots,b_{l_\Gamma}))
\end{equation}
is a rational function in variables $\{t_{i,j}\}$ which are in correspondence with the vertices of $\mathcal R$. However, we're interested in the result of applying the specialization $$t_{i,j}\longrightarrow x_i^{-1}x_{i+1}$$ to~(\ref{ipt}). We denote this specialization $F$, since it formally coincides with the specialization $F$ defined above when $i\in[1,n-1]$ and $j\in[1,n-i]$. We denote the obtained rational function in variables $\{x_i\}$ simply $\psi_\Gamma(b_1,\ldots,b_{l_\Gamma})$.

(Note that for any array $s=(s_{i,j})$ with a finite number of nonzero elements the power of $x_i$ monomial $F(e^s)$ is the sum of the elements of $s$ in row $i-1$ minus the sum of its elements in row $i$.) 

First of all, its worth mentioning that the functions $\psi_\Gamma(b_1,\ldots,b_{l_\Gamma})$ are well-defined, i.e. the reduced denominator of~(\ref{ipt}) does not vanish under $F$. To see this for any edge $e$ of $C$ consider the subgraph $\Delta_e$ and let $\varepsilon$ be the direction vector of $e$. Proposition~\ref{faces} shows that $\Delta_e$ contains exactly one component without a vertex in the top row, let $r$ be the row containing the single top vertex of that component. One may easily deduce that $F(e^\varepsilon)$ contains a nonzero power of $x_r$ and then invoke the remark at the end of the proof of Theorem~\ref{wbrion}.

Now we are ready to present the statement which will turn out to be the key to the proof of part b) of Theorem~\ref{contrib}.
\begin{theorem}\label{zero}
If $\Gamma$ is an ordinary subgraph and for some $i\ge a_\Gamma$ the number of its vertices in row $i+1$ is greater than in row $i$, then $\psi_\Gamma(b_1,\ldots,b_{l_\Gamma})=0$ for any integers $b_1\ge\ldots\ge b_{l_\Gamma}$.
\end{theorem}

Our proof of this Theorem requires an identity which relates the singular case of the $b_i$ all being the same to the regular case of them being pairwise distinct. Note that $D_\Gamma(b,\ldots,b)$ is a cone, we denote the vertex of this cone $v_\Gamma(b)$.
\begin{lemma}\label{singular}
For pairwise distinct $b_1>\ldots>b_{l_\Gamma}$ let $v_1,\ldots,v_m$ be the vertices of $D_\Gamma(b_1,\ldots,b_{l_\Gamma})$ with tangent cones $C_1,\ldots,C_m$. Then we have $$[l_\Gamma]_t!\sigma_{\varphi_\Gamma(b,\ldots,b)}(D_\Gamma(b,\ldots,b))=\sum\limits_{i=1}^m e^{v_\Gamma(b)-v_i}\sigma_{\varphi_\Gamma(b_1,\ldots,b_{l_\Gamma})}(C_i)$$ (the summands on the right are simply IPT's of the cones $C_i$ shifted by $v_\Gamma(b)-v_i$).
\end{lemma}
This identity is obtained as the weighted Brion theorem applied to $D_\Gamma(b,\ldots,b)$ viewed as a degeneration $D_\Gamma(b_1,\ldots,b_{l_\Gamma})$. We thus postpone the proof of the lemma until we have discussed these topics in detail.

\begin{proof}[Proof of Theorem~\ref{zero}.] We proceed by induction on the number of vertices in $\Gamma$ considering three cases. 

{\it Case 1.} No row in $\Gamma$ contains more than two vertices. This will include the base of our induction. Unfortunately, this case is the most computational part of the paper, although, in its essence, the argument is pretty straightforward.

First of all, if we have an $i>a_\Gamma$ such that $\Gamma$ has one vertex in row $i$ and two vertices in row $i+1$, we may apply the induction hypothesis. To do this, denote $\Gamma'$ the graph obtained from $\Gamma$ by removing all vertices in rows above $i$. Now consider a section of $D_\Gamma(b_1,\ldots,b_{l_\Gamma})$ obtained by fixing all coordinates in rows $i$ and above. The contribution of any such section to $\psi_\Gamma(b_1,\ldots,b_{l_\Gamma})$ is zero by the induction hypothesis applied to $\Gamma'$.

Thus we may assume that $\Gamma$ has one vertex in row $a_\Gamma$ and two vertices in row $a_\Gamma+1$. Figure~\ref{ex2} provides an example of such a graph.
We may also assume that $b_1=0$ since any $\psi_\Gamma(b)$ is obtained from $\psi_\Gamma(0)$ by multiplication by a monomial. We will compute $\psi_\Gamma(0)$ by considering the sections of $D_\Gamma(0)$ obtained by fixing the two coordinates in row $a_\Gamma+1$. If $\Gamma'$ is $\Gamma$ with the top vertex removed we have 
\begin{equation}\label{sectsum}
\psi_\Gamma(0)=\sum\limits_{b_1\ge0,b_2\le0}c_{b_1,b_2}\psi_{\Gamma'}(b_1,b_2),
\end{equation}
where
$$
c_{b_1,b_2}=
\begin{cases}
(1-t)^2& \text{if }b_1>0>b_2,\\
(1-t)& \text{if }b_1>0=b_2\text{ or }b_1=0>b_2,\\
(1-t^2)& \text{if }b_1=b_2=0.
\end{cases}
$$
Of course,~(\ref{sectsum}) needs to be formalized in order to make sense. This is done routinely so we will not go into detail. The idea is to observe that all the functions $\psi_{\Gamma'}(b_1,b_2)$ together with $\psi_\Gamma(0)$ have a common finite denominator (this is shown below). Multiplying~(\ref{sectsum}) by that common denominator yields an identity of formal Laurent series. 

Now, for some $b_1>b_2$ consider the vertices of $D_{\Gamma'}(b_1,b_2)$. It easy to see that among the corresponding subgraphs $\Delta_v$ there are exactly two consisting of two path graph components. Here are these two subgraphs for the example in Figure~\ref{ex2}.

\begin{picture}(95,35)
\put(35,30){\circle*{0.4}}
\put(30,25){\line(-1,-1){5}} \multiput(30,25)(1,-1){5}{\line(1,-1){0.2}} \put(40,25){\line(-1,-1){5}}
\put(25,20){\line(1,-1){5}} \multiput(35,20)(-1,-1){5}{\line(-1,-1){0.2}} 
\put(30,15){\line(1,-1){5}}
\put(30,5){$\Delta_{v_1}$}

\put(80,30){\circle*{0.4}}
\put(75,25){\line(-1,-1){5}} \multiput(75,25)(1,-1){5}{\line(1,-1){0.2}} \put(85,25){\line(-1,-1){5}}
\multiput(70,20)(1,-1){5}{\line(1,-1){0.2}} \put(80,20){\line(-1,-1){5}} 
\put(75,15){\line(1,-1){5}}
\put(75,5){$\Delta_{v_2}$}
\end{picture}

Denote $r$ the least number such that $\Gamma$ has one vertex in row $r$ but two vertices in row $r-1$. The difference between the two graphs is then that in one case (vertex $v_1$) the vertex in row $r$ is connected to its upper-left neighbor and in the other case (vertex $v_2$) to its upper-right neighbor.
\begin{proposition}
Consider a vertex $v$ of $D_{\Gamma'}(b_1,b_2)$ other than $v_1$ and $v_2$. Let $C_v$ be the tangent cone. The induction hypothesis than implies $$F(\sigma_{\varphi_{\Gamma'}(b_1,b_2)}(C_v))=0.$$
\end{proposition}
\begin{proof}
Consider the two components $\Gamma_1$ and $\Gamma_2$ of $\Delta_v$. We have $$F(\sigma_{\varphi_{\Gamma'}(b_1,b_2)}(C_v))=\psi_{\Gamma_1}(b_1)\psi_{\Gamma_2}(b_2).$$ The fact that at least one of $\Gamma_1$ and $\Gamma_2$ is not a path graph translates into that component containing one vertex in some row $i$ and two vertices in row $i+1$. The induction hypothesis then shows that the corresponding factor in the right-hand side above is zero.
\end{proof}

The weighted Brion theorem for $D_{\Gamma'}(b_1,b_2)$ is now seen to provide $$\psi_{\Gamma'}(b_1,b_2)=F(\sigma_{\varphi_{\Gamma'}(b_1,b_2)}(C_1))+F(\sigma_{\varphi_{\Gamma'}(b_1,b_2)}(C_2)),$$ where $C_1$ and $C_2$ are the corresponding tangent cones. It isn't too hard to compute the two summands on the right explicitly which is exactly what do.

Both of $C_1$ and $C_2$ cones are simplicial and unimodular. This is seen by considering the minimal integer direction vectors (generators) of their edges.  If $d_\Gamma$ is the last row containing vertices of $\Gamma$, then the set of generators for each of $C_1$ and $C_2$ satisfies the following description.
\begin{proposition}
The values of the coordinates of any such generator take only two values: 0 and either $-1$ or 1. For any $i\in[a_\Gamma+2,r-1]$ there is a single generator with exactly one nonzero coordinate in each of the rows in $[i,r-1]$ and all other coordinates 0. Also, for any $i\in[a_\Gamma+2,d_\Gamma]$ there is a single generator with exactly one nonzero coordinate in each of the rows in $[i,d_\Gamma]$ and all other coordinates 0.
\end{proposition}
\begin{proof}
In accordance with Proposition~\ref{faces}, for an edge $e$ of $C_1$ or $C_2$ the graph $\Delta_e$ is obtained from respectively $\Delta_{v_1}$ or $\Delta_{v_2}$ by deleting a single edge. This leaves $\Delta_e$ with exactly one component (of three) not containing a vertex in row $a_\Gamma+1$. The corresponding direction vector is obtained by setting the coordinates in this component to either $-1$ or 1 depending on the orientation of the deleted edge. All the other coordinates are zero.

The proposition now follows if we consider such vectors for each of the edges of $\Delta_{v_1}$ and $\Delta_{v_2}$ being deleted.
\end{proof}

We denote the generators described in the second sentence of the Proposition by $\varepsilon_i^1$ or $\varepsilon_i^2$ respectively. For the generators described in the third sentence we use the notations $\xi_i^1$ and $\xi_i^2$. Here are some of these vectors for our example with the cross marking the edge being deleted.

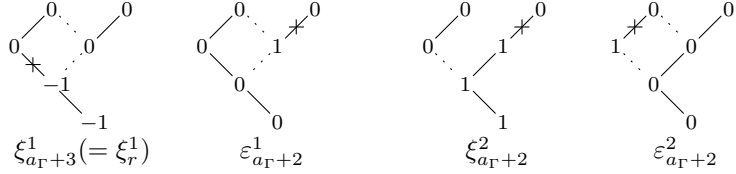
\begin{picture}(105,30)
\put(9.25,24){\footnotesize 0} \put(19.25,24){\footnotesize 0}
\put(9,24){\line(-1,-1){3}} \multiput(11,24)(1,-1){3}{\line(1,-1){0.2}} \put(19,24){\line(-1,-1){3}}
\put(4.25,19){\footnotesize 0} \put(14.25,19){\footnotesize 0}
\put(6,19){\line(1,-1){3}} \multiput(14,19)(-1,-1){3}{\line(-1,-1){0.2}} 
\put(6.5,17.5){\line(1,0){2}} \put(7.5,18.5){\line(0,-1){2}}
\put(8.75,14){\footnotesize $-1$}
\put(11,14){\line(1,-1){3}}
\put(13.75,9){\footnotesize $-1$}
\put(5,5){$\xi_{a_\Gamma+3}^1(=\xi_r^1)$}

\put(34.25,24){\footnotesize 0} \put(44.25,24){\footnotesize 0}
\put(34,24){\line(-1,-1){3}} \multiput(36,24)(1,-1){3}{\line(1,-1){0.2}} \put(44,24){\line(-1,-1){3}}
\put(41.5,22.5){\line(1,0){2}} \put(42.5,23.5){\line(0,-1){2}}
\put(29.25,19){\footnotesize 0} \put(39.25,19){\footnotesize 1}
\put(31,19){\line(1,-1){3}} \multiput(39,19)(-1,-1){3}{\line(-1,-1){0.2}} 
\put(34.25,14){\footnotesize 0}
\put(36,14){\line(1,-1){3}}
\put(39.25,9){\footnotesize 0}
\put(35,5){$\varepsilon_{a_\Gamma+2}^1$}

\put(64.25,24){\footnotesize 0} \put(74.25,24){\footnotesize 0}
\put(64,24){\line(-1,-1){3}} \multiput(66,254)(1,-1){3}{\line(1,-1){0.2}} \put(74,24){\line(-1,-1){3}}
\put(71.5,22.5){\line(1,0){2}} \put(72.5,23.5){\line(0,-1){2}}
\put(59.25,19){\footnotesize 0} \put(69.25,19){\footnotesize 1}
\multiput(61,19)(1,-1){3}{\line(1,-1){0.2}} \put(69,19){\line(-1,-1){3}}
\put(64.25,14){\footnotesize 1} 
\put(66,14){\line(1,-1){3}}
\put(69.25,9){\footnotesize 1}
\put(65,5){$\xi_{a_\Gamma+2}^2$}

\put(89.25,24){\footnotesize 0} \put(99.25,24){\footnotesize 0}
\put(89,24){\line(-1,-1){3}} \multiput(91,24)(1,-1){3}{\line(1,-1){0.2}} \put(99,24){\line(-1,-1){3}}
\put(86.5,22.5){\line(1,0){2}} \put(87.5,23.5){\line(0,-1){2}}
\put(84.25,19){\footnotesize 1} \put(94.25,19){\footnotesize 0}
\multiput(86,19)(1,-1){3}{\line(1,-1){0.2}} \put(94,19){\line(-1,-1){3}} 
\put(89.25,14){\footnotesize 0} 
\put(91,14){\line(1,-1){3}}
\put(94.25,9){\footnotesize 0}
\put(90,5){$\varepsilon_{a_\Gamma+2}^2$}
\end{picture}

It is easily seen that for all $i\in[a_\Gamma+2,r-1]$ we have $F(e^{\varepsilon_i^1})=F(e^{\varepsilon_i^2})$ and for all $i\in[a_\Gamma+2,d_\Gamma],i\neq r$ we have $F(e^{\xi_i^1})=F(e^{\xi_i^2})$. However, $F(e^{\xi_r^1})=x_r x_{d_\Gamma+1}^{-1}$ while $F(e^{\xi_r^2})=x_r^{-1} x_{d_\Gamma+1}$.

The last nuance we need to discuss to write out $\psi_{\Gamma'}(b_1,b_2)$ is how $\varphi_{\Gamma'}(b_1,b_2)$ behaves on faces of $C_1$ and $C_2$. This behavior is rather simple. 
\begin{proposition}
For a face $f$ of either cone we have $$\psi_{\Gamma'}(b_1,b_2)(f)=(1-t)^{\dim f}.$$
\end{proposition}
\begin{proof}
Since the graph $\Delta_f$ has $\dim f+2$ connected components, it is obtained from respectively $\Delta_{v_1}$ or $\Delta_{v_2}$ by deleting $\dim f$ edges. The definition of $\varphi_{\Gamma'}(b_1,b_2)$ then immediately provides the weight $(1-t)^{\dim f}$.
\end{proof}

The above facts give us $$F(\sigma_{\varphi_{\Gamma'}(b_1,b_2)}(C_1))=F(e^{v_1})\frac{1-t x_r x_{d_\Gamma+1}^{-1}}{1-x_r x_{d_\Gamma+1}^{-1}}Z$$ and $$F(\sigma_{\varphi_{\Gamma'}(b_1,b_2)}(C_2))=F(e^{v_2})\frac{1-t x_r^{-1} x_{d_\Gamma+1}}{1-x_r^{-1} x_{d_\Gamma+1}}Z,$$ where $$Z=F\left(\prod\limits_{i=a_\Gamma+2}^{r-1}\frac{1-te^{\varepsilon_i^1}}{1-e^{\varepsilon_i^1}} \prod\limits_{\substack{i\in[a_\Gamma+2,d_\Gamma]\\i\neq r}} \frac{1-te^{\xi_i^1}}{1-e^{\xi_i^1}} \right).$$

Also, we can now employ Lemma~\ref{singular} to derive 
\begin{multline*}
\psi_{\Gamma'}(0,0)=F(e^{v_{\Gamma'}(0)})\frac1{1+t}\left(\frac{1-t x_r x_{d_\Gamma+1}^{-1}}{1-x_r x_{d_\Gamma+1}^{-1}}+\frac{1-t x_r^{-1} x_{d_\Gamma+1}}{1-x_r^{-1} x_{d_\Gamma+1}}\right)Z=\\F(e^{v_{\Gamma'}(0)})Z.
\end{multline*}

Since $F(e^{v_1})=x_{a_\Gamma+1}^{-b_1-b_2}x_r^{b_2}x_{d_\Gamma+1}^{b_1}$ and $F(e^{v_2})=x_{a_\Gamma+1}^{-b_1-b_2}x_r^{b_1}x_{d_\Gamma+1}^{b_2}$ and \\$e^{v_{\Gamma'}(0)}=1$ we, conclusively, have
\begin{multline*}
\frac1Z\sum\limits_{b_1\le0,b_2\ge0}c_{b_1,b_2}\psi_{\Gamma'}(b_1,b_2)=\\
\frac{1-t x_r x_{d_\Gamma+1}^{-1}}{1-x_r x_{d_\Gamma+1}^{-1}}\left((1-t)^2\sum\limits_{b_1>0>b_2}x_{a_\Gamma+1}^{-b_1-b_2}x_r^{b_2}x_{d_\Gamma+1}^{b_1}+\right.\\
\left.(1-t)\sum\limits_{b_1>0}x_{a_\Gamma+1}^{-b_1}x_{d_\Gamma+1}^{b_1}+(1-t)\sum\limits_{b_2<0}x_{a_\Gamma+1}^{-b_2}x_r^{b_2}\right)+\\
\frac{1-t x_r^{-1} x_{d_\Gamma+1}}{1-x_r^{-1}x_{d_\Gamma+1}}\left((1-t)^2\sum\limits_{b_1>0>b_2}x_{a_\Gamma+1}^{-b_1-b_2}x_r^{b_1}x_{d_\Gamma+1}^{b_2}+\right.\\
\left.(1-t)\sum\limits_{b_1>0}x_{a_\Gamma+1}^{-b_1}x_r^{b_1}+(1-t)\sum\limits_{b_2<0}x_{a_\Gamma+1}^{-b_2}x_{d_\Gamma+1}^{b_2}\right)+1-t^2=\\
\frac{1-t x_r x_{d_\Gamma+1}^{-1}}{1-x_r x_{d_\Gamma+1}^{-1}}\left((1-t)^2\frac{x_r^{-1}x_{d_\Gamma+1}}{(1-x_{a_\Gamma+1}^{-1}x_{d_\Gamma+1})(1-x_{a_\Gamma+1}x_r^{-1})}+\right.\\
\left.(1-t)\frac{x_{a_\Gamma+1}^{-1}x_{d_\Gamma+1}}{1-x_{a_\Gamma+1}^{-1}x_{d_\Gamma+1}}+(1-t)\frac{x_{a_\Gamma+1}x_r^{-1}}{1-x_{a_\Gamma+1}x_r^{-1}}\right)+\\
\frac{1-t x_r^{-1} x_{d_\Gamma+1}}{1-x_r^{-1}x_{d_\Gamma+1}}\left((1-t)^2\frac{x_r x_{d_\Gamma+1}^{-1}}{(1-x_{a_\Gamma+1}^{-1}x_r)(1-x_{a_\Gamma+1}x_{d_\Gamma+1}^{-1})}+\right.\\
\left.(1-t)\frac{x_{a_\Gamma+1}^{-1}x_r}{1-x_{a_\Gamma+1}^{-1}x_r}+(1-t)\frac{x_{a_\Gamma+1}x_{d_\Gamma+1}^{-1}}{1-x_{a_\Gamma+1}x_{d_\Gamma+1}^{-1}}\right)+1-t^2\mathbf{=0}.
\end{multline*} 
The last equality is verified directly, best on a machine.
\\

{\it Case 2.} There exist at least two distinct $b_i$ (and we are not within case 1).

It suffices to show that for any vertex $v$ of $D_\Gamma(b_1,\ldots,b_{l_\Gamma})$ with tangent cone $C_v$ the contribution $F(\sigma_{\varphi_\Gamma(b_1,\ldots,b_{l_\Gamma})}(C_v))$ is zero.

By Proposition~\ref{faces} the number of connected components in $\Delta_v$ is the number of distinct $b_i$. Let $G_1,\ldots,G_m$ be these components, with $G_i$ containing the $l_i$-th through $r_i$-th vertex from the top row of $\Gamma$. We have the decomposition $$F(\sigma_{\varphi_\Gamma(b_1,\ldots,b_{l_\Gamma})}(C_v))=\prod\limits_{i=1}^m \psi_{G_i}(b_{l_i},\ldots,b_{r_i}).$$ That is because the cone $C_v$ is (a translate of) the direct sum of the cones $D_{G_i}(b_{l_i},\ldots,b_{r_i})$ and for a face $f=\bigoplus_{i=1}^m f_i$ of $C_v$ we have $$\varphi_\Gamma(b_1,\ldots,b_{l_\Gamma})=\prod\limits_{i=1}^m \varphi_{G_i}(b_{l_i},\ldots,b_{r_i})(f_i).$$ However, with the induction hypothesis taken into account, it is clear that at least one of the factors $\psi_{G_i}(b_{l_i},\ldots,b_{r_i})$ is zero.
\\

{\it Case 3.} We have $b_1=b_{l_\Gamma}$ (and we are not within case 1).

Consider any integers $b'_1>\ldots>b'_{l_\Gamma}$ and let $v_1,\ldots,v_m$ be the vertices of $D_\Gamma(b'_1,\ldots,b'_{l_\Gamma})$ with the tangent cones being $C'_1,\ldots,C'_m$. Now Lemma~\ref{singular} in combination with the argument for case 2 show that $$\psi_\Gamma(b_1,\ldots,b_{l_\Gamma})=\frac1{[l_\Gamma]_t!}F\left(\sum\limits_{i=1}^m e^{v_\Gamma(b_1)-v_i}\sigma_{\varphi_\Gamma(b'_1,\ldots,b'_{l_\Gamma})}(C'_i)\right)=0.$$
\end{proof}

\section{Weighted Brion's Theorem for Degenerated Polyhedra}

In order to prove Lemma~\ref{singular} it turns out necessary to occupy ourselves with the following question: how does our weighted version of Brion's theorem behave when we degenerate a polyhedron by shifting some of its facets? Let us elaborate.


We start with the following definition: two polytopes in $\mathbb{R}^d$ are said to be {\it analogous} if their normal  fans coincide. The defintion of the normal fan of a polytope (also referred to as the "polar fan" or the "dual fan") may, for example, be found in any textbook on toric geometry. In other words, two polytopes are analogous if there is a combinatorial equivalence between them such that the tangent cones at two corresponding faces may be obtained from one another by a translation.

We then say that a polytope $\Sigma'\subset\mathbb{R}^d$ is a degeneration of polytope $\Sigma\subset\mathbb{R}^d$ if there is continuous deformation $\Sigma(\alpha), \alpha\in[0,1]$ such that
\begin{enumerate}
\item $\Sigma(0)=\Sigma$,
\item $\Sigma(1)=\Sigma'$ and
\item $\Sigma(\alpha)$ is a polytope analogous to $\Sigma$ for $0\le\alpha<1$.
\end{enumerate}
One may thus say that we deform $\Sigma$ by continuously shifting its facets (or, rather, the hyperplanes containing its facets) in such a way that the combinatorial structure does not change until we reach point 1 in time.

It is easy to show that if $\Sigma'$ is a degeneration of $\Sigma$, then the normal fan of $\Sigma$ is a refinement of the normal fan of $\Sigma'$. This gives us a map $\pi:\mathcal F_\Sigma\rightarrow\mathcal F_{\Sigma'}$ sending $f\in\mathcal F_\Sigma$ to the minimal face $\pi(f)$ of $\Sigma'$ such that the cone corresponding to $\pi(f)$ in the normal fan of $\Sigma'$ contains the cone corresponding to $f$ in the normal fan of $\Sigma$. The map $\pi$ is surjective and has the property $\dim f\ge\dim \pi(f)$.


A useful fact is that, in terms of Brion's theorem, we may ignore the combinatorial structure having changed as a result of the degeneration. That is to say that the following identity holds. $$\sigma(\Sigma')=\sum\limits_{v\text{ vertex of }\Sigma}e^{\pi(v)-v}\sigma(C_v),$$ where $C_v$ is the corresponding tangent cone and we abuse the notations somewhat, knowing that $\pi(v)$ is a vertex.

We now demonstrate how and why this can be generalized to the weighted setting.
\begin{lemma}\label{wdegen}
In the above setting consider a weight function $\varphi:\mathcal F_\Sigma\rightarrow R$ for some commutative ring R. Next define $\varphi':\mathcal F_{\Sigma'}\rightarrow R$ by $$\varphi'(f')=\sum\limits_{f\in\pi^{-1}(f')}(-1)^{\dim f-\dim f'}\varphi(f).$$ Then the identity $$\sigma_{\varphi'}(\Sigma')=\sum\limits_{v\text{ vertex of }\Sigma}e^{\pi(v)-v}\sigma_\varphi(C_v)$$ holds.
\end{lemma}
\begin{proof}
It suffices to show that for any vertex $v'$ of $\Sigma'$ with tangent cone $C_{v'}$ we have $$\sigma_{\varphi'}(C_{v'})=\sum\limits_{\substack{\text{vertex }v\\\pi(v)=v'}}e^{v'-v}\sigma_\varphi(C_v).$$

Consider a face $f$ of $\Sigma$ such that $\pi(f)$ contains $v'$. Let $v_1,\ldots,v_m$ be the vertices of $f$ with $\pi(v_i)=v'$. Let $C_i$ denote the face of $C_{v_i}$ corresponding to (containing) $f$ and let $C'$ be the face of $C_{v'}$ corresponding to $\pi(f)$. We have 
\begin{multline*}
\sum\limits_{i=1}^m\sigma(\mathrm{Int}(C_i-v_i+v'))=(-1)^{\dim f}\sum\limits_{i=1}^m\sigma(-(C_i-v_i)+v'))=\\(-1)^{\dim f}\sigma(-(C'-v')+v')=(-1)^{\dim f-\dim\pi(f)}\sigma(\mathrm{Int}(C'),
\end{multline*}
where $\mathrm{Int}$ denotes the relative interior of a polyhedron (the polyhedron minus its boundary), $X+a$ is set $X$ translated by vector $a$ and $-X$ is $X$ reflected in the origin. The first an third equalities are due to Stanley reciprocity (see~\cite{beckrob}) while the second one is Brion's theorem for the cone $-(C'-v')+v'$ viewed as a degeneration of the polyhedron $-\bigcap_{i=1}^m C_i$. We understand the relative interior of a single point to be itself (rather than the empty set).

Now it remains to point out that adding up the above equalities with coefficients $\varphi(f)$ yields the desired identity.
\end{proof}

\section{Proof of Lemma~\ref{singular}}

It is actually somewhat more convenient to prove a generalization of Lemma~\ref{singular}.

For an ordinary graph $\Gamma$ let $b_1,\ldots,b_{l_\Gamma}$ be a strictly decreasing sequence of integers and $b'_1,\ldots,b'_{l_\Gamma}$ be decreasing but not strictly, i.e. at least two of $b'_i$ coincide. Specifically, let there be $m$ distinct $b'_i$ with the $j$-th largest of those $m$ values occurring $l_i$ times.

The polyhedron $D_\Gamma(b'_1,\ldots,b'_{l_\Gamma})$ is a degeneration of the polyhedron $D_\Gamma(b_1,\ldots,b_{l_\Gamma})$ in the above sense. To see this simply consider continuous functions $b_i(\alpha)$ with $b_i(0)=b_i$, $b_i(1)=b'_i$ and $b_1(\alpha)<\ldots<b_{l_\Gamma}(\alpha)$ for $\alpha<1$. Let $$\pi:\mathcal F_{D_\Gamma(b_1,\ldots,b_{l_\Gamma})}\rightarrow \mathcal F_{D_\Gamma(b'_1,\ldots,b'_{l_\Gamma})}$$ be the corresponding map. Also, let $v_1,\ldots,v_N$ be the vertices of $D_\Gamma(b_1,\ldots,b_{l_\Gamma})$ with respective tangent cones $C_1,\ldots,C_N$. The mentioned generalization is then as follows.
\begin{lemma}\label{gensingular}
$$[l_1]_t!\ldots[l_m]_t!\sigma_{\varphi_\Gamma(b'_1,\ldots,b'_{l_\Gamma})}\left(D_\Gamma(b'_1,\ldots,b'_{l_\Gamma})\right)=\sum\limits_{i=1}^N e^{\pi(v_i)-v_i}\sigma_{\varphi_\Gamma(b_1,\ldots,b_{l_\Gamma})}(C_i).$$
\end{lemma}

However, with Lemma~\ref{wdegen} taken into account, Lemma~\ref{gensingular} is an immediate consequence of the below fact.
\begin{lemma}\label{graphsum}
For a face $f$ of $D_\Gamma(b'_1,\ldots,b'_{l_\Gamma})$ we have $$[l_1]_t!\ldots[\l_m]_t!\varphi_\Gamma(b'_1,\ldots,b'_{l_\Gamma})(f)=\sum\limits_{g\in\pi^{-1}(f)}(-1)^{\dim g-\dim f}\varphi_\Gamma(b_1,\ldots,b_{l_\Gamma})(g).$$
\end{lemma}
\begin{proof}
First of all, let us describe the map $\pi$ in terms of corresponding subgraphs. Consider a face $g$ of $D_\Gamma(b_1,\ldots,b_{l_\Gamma})$.

\begin{proposition}
The subgraph $\Delta_{\pi(g)}$ is the smallest subgraph containing all edges of $\Delta_g$ and indeed corresponding to some face of $D_\Gamma(b'_1,\ldots,b'_{l_\Gamma})$, i.e satisfying the respective three conditions from Proposition~\ref{faces}.
\end{proposition}
\begin{proof}
The tangent cone $C_g$ at $g$ consists of points $x$ for which all coordinates outside of $\Gamma$ are 0, the coordinates in row $l_\Gamma$ are equal to $b_1,\ldots,b_{l_\Gamma}$ and for any edge of $\Delta_g$ the corresponding inequality between coordinates of $x$ holds. For a face $h$ of $D_\Gamma(b'_1,\ldots,b'_{l_\Gamma})$ the tangent cone $C_h$ is described analogously.

The cone in the normal fan corresponding to $h$ containing the cone in the normal fan corresponding to $g$ is equivalent to $C_g-x_g$ containing $C_h-x_h$, where $x_g$ is an arbitrary point in $g$ and $x_h$ is an arbitrary point in $h$. Let $e$ be an edge of $\Delta_g$ not in $\Delta_h$. For any point $x\in C_g-x_g$ the inequality corresponding to $e$ holds (since the corresponding two coordinates of $x_g$ are equal). However, we may find a point $y\in C_h-x_h$ for which the inequality corresponding to the edge does not hold.

This shows that any edge of $\Delta_g$ must be an edge of $\Delta_{\pi_g}$ and the minimality of $\Delta_\pi(g)$ follows from the minimality of normal fan cone corresponding to $\pi(g)$. 
\end{proof}

We prove the Lemma by induction on the number of vertices in $\Gamma$. The base of the induction is the case of $\Gamma$ having three vertices, two in row $a_\Gamma$ and one in row $a_\Gamma+1$. In this case we are dealing with a segment degenerating into a point and the Lemma simply states that $1+1-(1-t)=[2]_t!\cdot 1=1+t$. We turn to the step of our induction, it being broken up into two cases.

{\it Case 1.} The graph $\Delta_f$ is not connected.

Let $G_1,\ldots,G_m$ be the connected components of $\Delta_f$ which contain vertices from the top row $a_\Gamma$. Recall that the weight $\varphi(b'_1,\ldots,b'_{l_\Gamma})(f)$ is a product over the components in $\Delta_f$. Let $R$ be the product over the components other than these $G_j$.

The above characterization of $\pi$ shows that any component of $\Delta_f$ not amongst the $G_j$ is also a connected component of $\Delta_g$ for any $g\in\pi^{-1}(f)$. Thus $\varphi(b_1,\ldots,b_{l_\Gamma})(g)$ is a product of $R$ and factors corresponding to components of $\Delta_g$ which are contained in one of the $G_j$.

Write out the induction hypotheses for each of the degenerations of \\$D_{G_j}(b_1,\ldots,b_{l_j})$ into $D_{G_j}(b,\ldots,b)$ for some integer $b$. The observation in the previous paragraph shows that the product of these $m$ identities with an additional factor of $(-1)^{\dim f}R$ is precisely the desired identity. The induction hypothesis applies since in this case all the $G_i$ have less vertices than $\Gamma$.
\\

{\it Case 2.} The graph $\Delta_f$ is connected, i.e. $\Delta_f=\Gamma$. This means that \\$D_\Gamma(b'_1,\ldots,b'_{l_\Gamma})$ is a cone (all of the $b'_i$ are the same) and $f$ is the vertex of that cone.

Denote the value of all the $b'_i$ as $b$. The preimage $\pi^{-1}(f)$ consists precisely of the bounded faces of $D_\Gamma(b_1,\ldots,b_{l_\Gamma})$ because, for any degeneration, $\pi(f)$ is bounded if and only if $f$ is.
\begin{proposition}
A face $g$ of $D_\Gamma(b_1,\ldots,b_{l_\Gamma})$ is bounded if and only if $\Delta_g$ possesses the following two properties. 

Whenever both vertices $(i,j)$ and $(i+1,j-1)$ are the leftmost within $\Gamma$ in their respective rows, then $\Delta_g$ includes the edge joining them. 

Similarly, if both vertices $(i,j)$ and $(i+1,j)$ are the rightmost within $\Gamma$ in their respective rows, then $\Delta_g$ includes the edge joining them. 
\end{proposition}
\begin{proof}
If the conditions are satisfied, then, visibly, every coordinate of any point in $g$ is between $b_1$ and $b_{l_\Gamma}$. Conversely, if the first condition is violated, then $g$ contains points for which the coordinate $i+1,j-1$ is arbitrarily large. Similarly, if the second condition is violated, then $g$ contains points for which the negative of coordinate $i+1,j$ is arbitrarily large.
\end{proof}

Let $\Gamma'$ be $\Gamma$ with its top row removed. Choose $g$, a bounded face of \\$D_\Gamma(b_1,\ldots,b_{l_\Gamma})$, and let $\Delta'$ be obtained from $\Delta_g$ by removing the vertices in row $a_\Gamma$. The graph $\Delta'$ is a subgraph of $\Gamma'$.

Since all the vertices in the top row of $\Delta_g$ are in different components, every component of $\Delta'$ contains no more than two vertices from the top row of $\Delta'$. We introduce a nonincreasing sequence of integers $c'_1,\ldots,c'_{l_{\Gamma'}}$ such that $c'_i=c'_{i+1}$ if and only if vertices number $i$ and $i+1$ from the left in the top row of $\Delta'$ are in the same component.

On top of that, let $c_1,\ldots,c_{l_{\Gamma'}}$ be a strictly decreasing sequence of integers. We have three polyhedra: one is $D_{\Gamma'}(c_1,\ldots,c_{l_{\Gamma'}})$, the second is $D_{\Gamma'}(c'_1,\ldots,c'_{l_{\Gamma'}})$ and the third is the cone $D_{\Gamma'}(c,\ldots,c)$ (where $c$ is an arbitrary integer).

The second polyhedron is a degeneration of the first, while the cone is a degeneration of both others. We have the three corresponding maps of faces:
$$\pi':\mathcal F_{D_{\Gamma'}(c_1,\ldots,c_{l_{\Gamma'}})}\rightarrow\mathcal F_{D_{\Gamma'}(c,\ldots,c)},$$
$$\rho:\mathcal F_{D_{\Gamma'}(c_1,\ldots,c_{l_{\Gamma'}})}\rightarrow\mathcal F_{D_{\Gamma'}(c'_1,\ldots,c'_{l_{\Gamma'}})}\text{ and }$$
$$\upsilon:\mathcal F_{D_{\Gamma'}(c'_1,\ldots,c'_{l_{\Gamma'}})}\rightarrow\mathcal F_{D_{\Gamma'}(c,\ldots,c)}.$$

The induction hypothesis for the vertex $f'$ of $D_{\Gamma'}(c,\ldots,c)$ reads
\begin{equation}\label{hypoth1}
[l_{\Gamma'}]_t!\varphi_{\Gamma'}(c,\ldots,c)(f')=\sum\limits_{h\in(\pi')^{-1}(f')}(-1)^{\dim h}\varphi_{\Gamma'}(c_1,\ldots,c_{l_{\Gamma'}})(h).
\end{equation}

Further, $\Delta'$ corresponds to some face of $D_{\Gamma'}(c'_1,\ldots,c'_{l_{\Gamma'}})$ which we denote $g'$ (so $\Delta'=\Delta_{g'}$). Let $d$ denote the number of pairs $c'_i=c'_{i+1}$. The induction hypothesis applied to $g'$ states that 
\begin{equation}\label{hypoth2}
(1+t)^d\varphi_{\Gamma'}(c'_1,\ldots,c'_{l_{\Gamma'}})(g')=\sum\limits_{h\in\rho^{-1}(g')}(-1)^{\dim h-\dim g'}\varphi_{\Gamma'}(c_1,\ldots,c_{l_{\Gamma'}})(h).
\end{equation}

Now denote $I_g$ the graph obtained from $\Delta_g$ by removing all vertices below row $a_\Gamma+1$, that is, leaving only the top two rows. For bounded faces of $D_\Gamma(b_1,\ldots,b_{l_\Gamma})$ write $g_1\sim g_2$ if and only if $I_{g_1}=I_{g_2}$. The faces in the equivalence class of $g$ are in bijection with the bounded faces of $D_{\Gamma'}(c'_1,\ldots,c'_{l_{\Gamma'}})$, face $g_1$ corresponding to face $g'_1$ (defined analogously to $g'$).

The previous paragraph shows that adding up identities~(\ref{hypoth2}) for all $g_1\sim g$ gives
\begin{multline}\label{total}
(1+t)^d\sum\limits_{g'_1\in\upsilon^{-1}(f')}(-1)^{\dim g'_1}\varphi_{\Gamma'}(c'_1,\ldots,c'_{l_{\Gamma'}})(g'_1)=\\
\sum\limits_{h\in(\pi')^{-1}(f')}(-1)^{\dim h}\varphi_{\Gamma'}(c_1,\ldots,c_{l_{\Gamma'}})(h),
\end{multline}
$\upsilon^{-1}(f')$ being precisely the set of bounded faces of $D_{\Gamma'}(c'_1,\ldots,c'_{l_{\Gamma'}})$. The sum in the right-hand side ranges over all of $(\pi')^{-1}(f')$ because $\pi'=\upsilon\rho$.

Denote $e$ is the number of vertices in row $a_\Gamma+1$ in $\Delta_g$ which are not connected to any vertex from the top row $a_\Gamma$. For any $g_1\sim g$ we have
\begin{equation}\label{gviag}
\varphi_\Gamma(b_1,\ldots,b_{l_\Gamma})(g_1)=(1-t)^e(1-t^2)^d\varphi_{\Gamma'}(c'_1,\ldots,c'_{l_{\Gamma'}})(g'_1).
\end{equation}

What we do now is substitute~(\ref{hypoth1}) into~(\ref{total}) and then substitute~(\ref{gviag}) into the result. Taking into account that $\dim g_1=\dim g'_1+e$, we obtain
\begin{equation}\label{subst}
\sum\limits_{g_1\sim g}(-1)^{\dim g_1}\varphi_\Gamma(b_1,\ldots,b_{l_\Gamma})(g_1)=(-1)^e(1-t)^{d+e}[l_{\Gamma'}]_t!\varphi_{\Gamma'}(c,\ldots,c)(f').
\end{equation}
We denote $\nu(I_g)$ the coefficient $(-1)^e(1-t)^{d+e}$, both $d$ and $e$ being determined by $I_g$.

However, we have $$[l_\Gamma]_t!\varphi_\Gamma(b,\ldots,b)(f)=\kappa_\Gamma[l_{\Gamma'}]_t!\varphi_{\Gamma'}(c,\ldots,c)(f'),$$ where
\begin{equation}\label{fviaf}
\kappa_\Gamma=
\begin{cases}
\frac{1-t^{l_\Gamma}}{1-t}&\text{if } l_{\Gamma'}=l_\Gamma-1,\\
1& \text{if } l_{\Gamma'}=l_\Gamma,\\
1-t& \text{if } l_{\Gamma'}=l_\Gamma+1.
\end{cases}
\end{equation}

Therefore, if we sum up identity~(\ref{subst}) with $g$ ranging over a set $S$ of representatives for relation $\sim$, all that will be left to prove is the following proposition.
\begin{proposition}
$\sum\limits_{g\in S} \nu(I_g)=\kappa_\Gamma$
\end{proposition}
\begin{proof}
$I_g$ is a graph with two rows. Each vertex from the lower row is either connected to one of the two (or one) vertices directly above it or is isolated. The fact that $g$ is bounded translates into the following two additional requirements. If the leftmost vertex in the lower row has no upper-left neighbor, it is necessarily connected to its upper-right neighbor (i.e. it is not isolated). Similarly, if the rightmost vertex in the lower row has no upper-right neighbor, it is necessarily connected to its upper-left neighbor. Here are examples for each of the three cases from definition~(\ref{fviaf}).

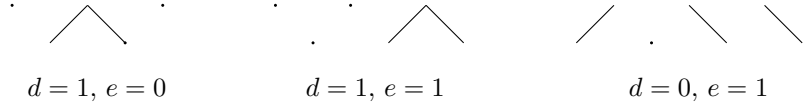
\begin{picture}(105,20)
\put(15,15){\line(-1,-1){5}} \put(15,15){\line(1,-1){5}}
\put(5,15){\circle*{0.4}} \put(25,15){\circle*{0.4}} \put(20,10){\circle*{0.4}}
\put(7,3){$d=1$, $e=0$}

\put(60,15){\line(-1,-1){5}} \put(60,15){\line(1,-1){5}}
\put(50,15){\circle*{0.4}} \put(40,15){\circle*{0.4}} \put(45,10){\circle*{0.4}}
\put(44,3){$d=1$, $e=1$}

\put(85,15){\line(-1,-1){5}} \put(95,15){\line(1,-1){5}} \put(105,15){\line(1,-1){5}}
\put(90,10){\circle*{0.4}}
\put(87,3){$d=0$, $e=1$}
\end{picture}

Coincidentally, in each of the three cases the number of different possible $I_g$ is $3^{l_\Gamma-1}$.

We denote
$$
\sum\limits_{g\in S} \nu(I_g)=
\begin{cases}
\Sigma_{l_\Gamma}^-&\text{if } l_{\Gamma'}=l_\Gamma-1,\\
\Sigma_{l_\Gamma}^0& \text{if } l_{\Gamma'}=l_\Gamma,\\
\Sigma_{l_\Gamma}^+& \text{if } l_{\Gamma'}=l_\Gamma+1.
\end{cases}
$$
The Proposition follows directly from the recurrence relations
$$\Sigma_{l+1}^-=-(1-t)\Sigma_{l}^-+\Sigma_{l}^-+\Sigma_{l}^0,$$
$$\Sigma_{l+1}^0=-(1-t)\Sigma_{l}^-+(1-t)\Sigma_{l}^-+\Sigma_{l}^0,$$
$$\Sigma_{l+1}^+=-(1-t)\Sigma_{l}^0+(1-t)\Sigma_{l}^0+\Sigma_{l}^+.$$
\end{proof}

We have completed the consideration of case 2 and subsequently the step of our induction. 
\end{proof}

\section{Application to the Finite Case}

With the above machinery at hand, little more effort is needed to prove Theorem~\ref{contribfin}. We give an outline of this argument, the details being filled in straightforwardly.

Let $\lambda$ be an integral dominant $\mathfrak{sl}_n$-weight. As mentioned above, the polytope $GT_\lambda$ is in a natural bijection with the polytope $D_{\mathcal T}(\lambda_1,\ldots,\lambda_{n-1},0)$. Moreover, for an integer point $A$ in $GT_\lambda$ we visibly have $$p_A=\varphi_{\mathcal T}(\lambda_1,\ldots,\lambda_{n-1},0)(f),$$ where $f$ is the minimal face containing $A$. We obtain $$\sum_{A\in\mathbf{GT}_\lambda}p_A e^{\mu_A}=\psi_{\mathcal T}(\lambda_1,\ldots,\lambda_{n-1},0)|_{x_0=x_n=1}.$$

Now, the distinguished set of vertices of $GT_\lambda$, mentioned in Theorem~\ref{contribfin} can be described in terms of $D_{\mathcal T}(\lambda_1,\ldots,\lambda_{n-1},0)$ as follows. They are those vertices $v$ for which $\Delta_v$ contains no component which has more vertices in some row $i$ than in row $i-1$ (with $i>0$). We term those vertices ``relevant", the rest being ``non-relevant".

Indeed, let the partition $(\lambda_1,\ldots,\lambda_{n-1},0)$ have type $l_1,\ldots,l_m$, i.e. $r$-th largest part occurs $l_r$ times. Choose a vertex $v$.  We see that $\Delta_v$ has $m$ connected components, denote them $\Gamma_1,\ldots,\Gamma_m$. If $C_v$ is the tangent cone to $D_{\mathcal T}(\lambda_1,\ldots,\lambda_{n-1},0)$, we have $$C_v=D_{\Gamma_1}(\lambda_1,\ldots,\lambda_{l_1})\times D_{\Gamma_2}(\lambda_{l_1+1},\ldots,\lambda_{l_1+l_2})\times\ldots$$ and, consequently, $$F(\sigma_{\varphi_{\mathcal T}(\lambda_1,\ldots,\lambda_{n-1},0)}(C_v))=\psi_{\Gamma_1}(\lambda_1,\ldots,\lambda_{l_1}) \psi_{\Gamma_2}(\lambda_{l_1+1},\ldots,\lambda_{l_1+l_2})\ldots$$

The contributions of non-relevant vertices being zero now follows from Theorem~\ref{zero}.

Now, let us consider the relevant vertices. In the case of $\lambda$ being regular the following facts can be easily deduced (and are found in~\cite{me1}). There are exactly $n!$ relevant vertices of $GT_\lambda$. For each $w\in W$ we have exactly one relevant vertex with $\mu_v=w\lambda$, denote this vertex $v_w$. The tangent cone at $v_w$ is simplicial and unimodular. Let $\varepsilon_{w,1},\ldots,\varepsilon_{w,{n\choose 2}}$ be the generators of edges of tangent cone $C_{v_w}$. The set  $$\left\{F(e^{\varepsilon_{w,1}}),\ldots,F\left(e^{\varepsilon_{w,{n\choose 2}}}\right)\right\}$$ coincides with the set $\{e^{-w\alpha},\alpha\in\Phi^+\}$. Finally, for a face $f$ of $C_{v_w}$ we have $\varphi(f)=(1-t)^{\dim f}$ (in the notations of Section~\ref{gtbrion}). 

All of this together translates into the formula for the contribution of a relevant vertex provided by Theorem~\ref{contribfin}.

If $\lambda$ is singular, we reduce to the regular case. Indeed, let $\lambda^1$ be some regular integral dominant weight. We see that $GT_{\lambda}$ is a degeneration of $GT_{\lambda^1}$. If $\pi$ is the corresponding map between face sets, we have $\pi(v_{w_1}^1)=\pi(v_{w_2}^1)$ (relevant vertices of $GT_{\lambda^1}$) if and only if $w_1\lambda=w_2\lambda$. Since our degeneration coincides with the degeneration of $D_{\mathcal T}(\lambda_1^1,\ldots,\lambda_{n-1}^1,0)$ into $D_{\mathcal T}(\lambda_1,\ldots,\lambda_{n-1},0)$, we may apply Lemma~\ref{gensingular} to show that for a vertex $v$ of $GT_\lambda$ we have $$F(\sigma_\varphi(C_v))=\frac 1{[l_1]_t!\ldots[l_m]_t!}\sum_{\pi(v_w^1)=v}F(e^{v-v_w^1}\sigma_{\varphi^1}(C_{v_w^1})).$$

With the regular case taken into account, the above identity proves Theorem~\ref{contribfin} for the case of singular $\lambda$.

The structure of the above argument is, in its essence, the same as that of the argument we give in Section~\ref{last} to prove Theorem~\ref{contrib}. However, significant care is needed to deal with the infinite dimension of the polyhedra, the tools necessary for that will be developed in the first two sections of Part~\ref{proof}.

We finish this part off by showing how applying Lemma~\ref{graphsum} in the above situation provides some identities in $\mathbb{Z}[t]$ which we find to be fascinating. Indeed, let $\lambda$ be some integral dominant $\mathfrak{sl}_n$-weight and let $\lambda^1$ be such a weight which is also regular. Apply Lemma~\ref{graphsum} to the degeneration of $D_{\mathcal T}(\lambda_1^1,\ldots,\lambda_{n-1}^1,0)$ into $D_{\mathcal T}(\lambda_1,\ldots,\lambda_{n-1},0)$ and then to the degeneration of $D_{\mathcal T}(\lambda_1^1,\ldots,\lambda_{n-1}^1,0)$ into $D_{\mathcal T}(0,\ldots,0)$ (which is a point). Combining the results provides 
\begin{theorem}
$$\sum_{\substack{f\text{ face}\\\text{of }GT_\lambda}}(-1)^{\dim f}\varphi(f)={{n\choose{l_1,\ldots,l_m}}\!}_t\,,$$ where $l_1,\ldots,l_m$ is the type of partition $(\lambda_1,\ldots,\lambda_{n-1},0)$ and we refer to the $t$-multinomial coefficient.
\end{theorem}
 In particular, when $\lambda$ is regular on the right-hand side we simply have $[n]_t!$.

\part{Structure of $\Pi$ and Proof of Theorem~\ref{contrib}}\label{proof}

For the entirety of this Part we consider $\lambda$ to be a fixed nonzero integral dominant $\widehat{\mathfrak sl}_n$-weight, $n$ also being fixed. All the definitions from the Part~\ref{part1} should be understood with respect to these values.

\section{The Brion-type Theorem for $\bar\Pi$}

Recall the infinite-dimensional polyhedron $\Pi$ introduced in Section~\ref{affbrion}.

We call nonempty $f\subset\Pi$ a face of $\Pi$ if it is the intersection of $\Pi$ and some of the spaces $$E_i=\{x|x_i=0\}\cap V$$ and $$H_i=\{x|\chi_i(x)=k\}\cap V.$$ The faces form a lattice with respect to inclusion, a vertex is any minimal element in this lattice.
\begin{proposition}\label{vert}
The vertices of $\Pi$ are precisely those points $x\in\Pi$ for which for any $i$ at least on of $x_i=0$ or $\chi_i(x)=0$ holds.
\end{proposition} 
\begin{proof}
Consider the face $$\Pi^l=\Pi\cap\bigcap_{i\le l}H_i.$$ In~\cite{me2} the statement of the Proposition was proved for vertices contained in $\Pi^0$. Since any $\Pi^l$ is obtained from $\Pi^0$ by the operator $(x_i)\rightarrow(x_{i+l})$, the Proposition also holds for vertices contained in any $\Pi^l$, but the $\Pi^l$ exhaust $\Pi$.
\end{proof}

We see that for $x\in V$ we have $s_{i,j}(x)=s_{i-1,j}(x)$ if and only if $x\in H_{in+j(n-1)}$ and $s_{i,j}(x)=s_{i-1,j+1}(x)$ if and only if $x\in E_{in+j(n-1)}$. This shows that the weight $p(x)$ depends only on the minimal face of $\Pi$ containing $x$. Every point is contained in some finite-dimensional face and every finite-dimensional face is a finite-dimensional polyhedron. Thus for any finite-dimensional face $f$ of positive dimension we may take a point $x$ in its relative interior and see that $f$ is the minimal face containing $x$. If we then define $\varphi(f)=p(x)$, we obtain a function $$\varphi:\mathcal F_\Pi\rightarrow\mathbb{Z}[t],$$ where $\mathcal F_\Pi$ is the set of all finite-dimensional faces of $\Pi$.

We now set out to define the series $\tau_{\bar v}$ mentioned in Section~\ref{affbrion}.

At any vertex we $v$ of $\Pi$ we have the tangent cone $$C_v=\{v+\alpha(x-v),x\in\Pi,\alpha\ge 0\}.$$ For any face $f$ containing $v$ we have the corresponding face of $C_v$: $$f_v=\{v+\alpha(x-v),x\in f,\alpha\ge 0\}.$$

For any edge (one-dimensional face) $e$ of $\Pi$ containing $v$ we have its generator, the minimal integer vector $\varepsilon$ such that $v+\varepsilon\in e$. Let $\{\varepsilon_{v,i},i>0\}$ be the set of generating vectors of all edges containing $v$. Any point of $C_v$ is obtained from $v$ by adding a non-negative linear combination of these $\varepsilon_{v,i}$.

We will make use of the following propositions.

For a Laurent monomial $y$ in $z_1,\ldots,z_{n-1},q$ denote $\deg y$ the power in which $y$ contains $q$. Also, recall the specialization $G$ given by~(\ref{affspec}).
\begin{proposition}\label{qbig}
For any vertex $v$ and any $N\in\mathbb{Z}$ there is only a finite number of $i$ such that $\deg G(e^{\varepsilon_{v,i}})<N$. 
\end{proposition}
\begin{proof}
We use the following fact. For any $M$ there is only a finite number of vertices $u$ with $\deg G(e^{\bar u})<M$. This, for example, follows from the fact that the sum of these monomials over all integer points in $\Pi$ (including all vertices) is $e^{-\lambda}\charac{L_\lambda}$.

Every edge of $\Pi$ is a segment joining two vertices. In other words, for every $\varepsilon_{v,i}$ there is a positive integer $K$ such that $v+K\varepsilon_{v,i}$ is some other vertex $u_i$. Let $l$ be the number of the first nonzero coordinate in $\varepsilon_{v,i}$ and let that coordinate be equal to $c$. We have $\chi_l(u_i)=\chi_l(v)+Kc$, which shows that $K\le k$.

However, $$\deg G(e^{\bar u_i})-\deg G(e^{\bar v})=K\deg G(e^{\varepsilon_{v,i}}).$$ Now we see that an infinite number of $\varepsilon_{v,i}$ with $\deg G(e^{\varepsilon_{v,i}})<N$ would contradict the fact in the beginning of the proof.
\end{proof}

\begin{proposition}\label{fundpar}
Consider any finite dimensional rational cone $C$ and map $\psi:\mathcal F_C\rightarrow R$ for some commutative ring $R$. Let $\varepsilon_1,\ldots,\varepsilon_m$ be the generators of the edges of $C$. Then 
\begin{equation}\label{numer}
(1-e^{\varepsilon_1})\ldots(1-e^{\varepsilon_m})S_\psi(C)
\end{equation}
is a linear combination of exponents of points of the form $$v+\alpha_1\varepsilon_1+\ldots+\alpha_m\varepsilon_m$$ with all $\alpha_i\in[0,1]$.
\end{proposition}
\begin{proof}
Consider a triangulation of $C$ by simplicial cones, each cone being generated by some of the $\varepsilon_i$. Let $T$ be a face of one of the cones. We may assume that $T$ is generated by $\varepsilon_1,\ldots,\varepsilon_l$. The expression $$(1-e^{\varepsilon_1})\ldots(1-e^{\varepsilon_l})S(\mathrm{Int}(T))$$ is precisely the sum of exponentials of all integer points within the parallelepiped $$\{v+\alpha_1\varepsilon_1+\ldots+\alpha_l\varepsilon_l,\alpha_i\in(0,1]\}.$$

Let $f$ be the minimal face of $C$ containing $T$. We see that $$\psi(f)(1-e^{\varepsilon_1})\ldots(1-e^{\varepsilon_m})S(\mathrm{Int}(T))$$ is a sum of exponentials of the desired type. However~(\ref{numer}) is the sum of the above expressions over all $T$ plus $\varphi(u)e^{u}$, where $u$ is the vertex of $C$.
\end{proof}

Let us denote $C_{\bar v}=\widebar{C_v}$.  Note that the generators of edges of $C_{\bar v}$ comprise the same set $\{\varepsilon_{v,i}\}$. In the below arguments we switch somewhat freely between $C_{\bar v}$ and $C_v$ and their attributes. The reader should be attentive not to miss the $\bar{}$  and keep in mind that in most ways the structure of these cones is the same.

In any point $x\in C_v$ we now may define the function $$p_v(x)=\varphi\left(\min_{x\in f_v}f\right).$$ We have the formal Laurent series in variables $t_i$ $$S_{\bar\varphi}(C_{\bar v})=\sum_{x\in C_{\bar v}\cap\mathbb{Z}^{{}^\infty}}\widebar{p_v}(x)e^x.$$

In what follows we implicitly use the fact that $G(e^{\varepsilon_{v,i}})\neq 1$. This will be proved in the next section.

Consider the cone $C_v-v=C_{\bar v}-\bar v$ with vertex at the origin. Just like for a finite dimensional cone, Laurent series that are sums of monomials $e^x$ with $x\in C_v-v$ comprise a ring. Both $e^{-v}S_{\bar\varphi}(C_{\bar v})$ and the product $$(1-e^{\varepsilon_{v,1}})(1-e^{\varepsilon_{v,2}})\ldots$$ are elements of that ring and thus the product $$Q_v=S_{\bar\varphi}(C_{\bar v})(1-e^{\varepsilon_{v,1}})(1-e^{\varepsilon_{v,2}})\ldots$$ is well-defined.

\begin{lemma}\label{welldef}
$G(Q_v)$ is a well-defined element of $\mathfrak S$.
\end{lemma}
\begin{proof}
We are to show that for any integer $N$ among those monomials $e^x$ that occur in $Q_v$ with a nonzero coefficient there is only a finite number for which $\deg G(e^x)<N$.

For $l\gg 0$ the intersection $$C_{v,l}=C_{\bar v} \cap\bigcap_{i<-l}\widebar{H_i}\cap\bigcap_{i>l}\widebar{E_i}$$ is a finite-dimensional cone with vertex $\bar v$ and is a face of $C_{\bar v}$. We thus have an increasing sequence of faces that exhausts $C_{\bar v}$. Every edge of cone $C_{v,l}$ is an edge of $C_{\bar v}$. Choose some cone $C_{v,l}$ and suppose that its edges are generated by $\varepsilon_{v,1},\ldots,\varepsilon_{v,m}$. We then denote $$Q_{v,l}=(1-e^{\varepsilon_{v,1}})\ldots(1-e^{\varepsilon_{v,m}})S_{\bar\varphi}(C_{v,l}),$$ where $\bar\varphi$ is evaluated in faces of $C_{v,l}$ in the natural way. Evidently, the coefficient of $e^x$ in $Q_{v,l}$ stabilizes onto the coefficient of $e^x$ in $Q_v$ as $l$ approaches infinity. We prove the lemma by showing that for $l\gg 0$ the difference $Q_{v,l}-Q_{v,l-1}$ has a zero coefficient at any monomial $e^x$ with $\deg G(e^x)<N$.

Let $S$ be the set of those $\varepsilon_{v,i}$ for which $\deg G(e^{\varepsilon_{v,i}})<0$ and let $$K=\deg\left(\prod_{\varepsilon_{v,i}\in S}G(e^{\varepsilon_{v,i}})\right).$$ We show that $Q_{v,l}-Q_{v,l-1}$ has a zero coefficient at any monomial $e^x$ with $\deg G(e^x)<N$ whenever the following holds. For every $\varepsilon_{v,i}$ which generates an edge contained in $C_{v,l}$ but not $C_{v,l-1}$ one has $\deg G(e^{\varepsilon_{v,i}})\ge N-K$. This visibly holds for all $l\gg 0$, fix some $l$ for which it does.

Proposition~\ref{fundpar} shows that for every $e^x$ which appears in $Q_{v,l}$ the vector $x$ is of the form $$\bar v+\alpha_1\varepsilon_{v,1}+\ldots+\alpha_m\varepsilon_{v,m},\alpha_i\in[0,1].$$ If, however, $e^x$ appears in $Q_{v,l}-Q_{v,l-1}$, then we must have $\alpha_i>0$ for some $\varepsilon_{v,i}$ which generates an edge contained in $C_{v,l}$ but not $C_{v,l-1}$. This is since $C_{v,l-1}$ is a face of $C_{v,l}$.

We fix $x$ such that $e^x$ appears in $Q_{v,l}-Q_{v,l-1}$ and $$x=\bar v+\alpha_1\varepsilon_{v,1}+\ldots+\alpha_m\varepsilon_{v,m},\alpha_i\in[0,1]$$ and show that $$\sum_{\substack{i\in[1,m],\\\bar v+\varepsilon_{v,i}\not\in C_{v,l-1}}}\alpha_i\ge 1.$$ This completes the proof since we have $$\deg G(e^x)\ge K+(N-K)\sum_{\substack{i\in[1,m],\\\bar v+\varepsilon_{v,i}\not\in C_{v,l-1}}}\alpha_i.$$

To prove this last assertion we use the following fact about the $\varepsilon_{v,i}$, which may be extracted from~\cite{me2}. All the nonzero coordinates (terms) of  $\varepsilon_{v,i}$ are either $-1$ or 1. 

From the fact that $\bar v\in C_{v,l}$ we see that $\bar v_i=0$ whenever $i<-l$ or $i>l$. Further we see that if $\varepsilon_{v,i}$ which generates an edge contained in $C_{v,l}$ but not $C_{v,l-1}$, then it has a nonzero coordinate with number either $-l-1$ or $l+1$. Moreover, the fact that $v\in H_{-l-1}$ and $v\in E_{l+1}$ implies the following. If the coordinate of $\varepsilon_{v,i}$ with number $-l-1$ is nonzero, then this coordinate must be $-1$ in order to have $\chi_{-l-1}(v+\varepsilon_{v,i})\le k$. Also, if the coordinate of $\varepsilon_{v,i}$ with number $l+1$ is nonzero, then this coordinate must be 1 in order for the corresponding coordinate of $v+\varepsilon_{v,i}$ to be nonnegative. Herefrom we deduce that if $$\sum_{\substack{i\in[1,m],\\\bar v+\varepsilon_{v,i}\not\in C_{v,l}}}\alpha_i<1,$$ then the coordinate of $x$ with number either $-l-1$ or $l+1$ turns out to be non-integral. 
\end{proof}

Now we can finally define 
\begin{equation}\label{tau}
\tau_{\bar v}=\frac{G(Q_v)}{(1-G(e^{\varepsilon_{v,1}}))(1-G(e^{\varepsilon_{v,2}}))\ldots}.
\end{equation}
Proposition~\ref{qbig} shows that the denominator is indeed an invertible element of $\mathfrak S$.

The proof above shows that $G(Q_v)$ contains no monomials with powers of $q$ less than $\deg G(e^{\bar v})+K$ (the number $K$ is defined in the proof). Also, it is obvious that the denominator of~(\ref{tau}) contains no monomials with powers of $q$ less than $K$. This shows that $\tau_{\bar v}$ only contains powers of $q$ no less than $\deg G(e^{\bar v})$.

Furthermore, consider a cone $C_{v,l}$ from the proof and let it be generated by $\varepsilon_1,\ldots,\varepsilon_m$. One also sees that $G(Q_{v,l})$ contains no monomials with powers of $q$ less than $\deg G(e^{\bar v})+K$ and $$G((1-e^{\varepsilon_{v,1}})\ldots(1-e^{\varepsilon_{v,m}}))$$ contains no monomials with powers of $q$ less than $K$. Consequently, we may view the quotient of $G(Q_{v,l})$ by the above product as $\tau_{\bar v,l}\in\mathfrak S$ which only contains powers of $q$ no less than $\deg G(e^{\bar v})$. (As a rational function this quotient is, of course, $G(\sigma_{\bar\varphi}(C_{v,l}))$.)

These observations are necessary to obtain the goal of this section.
\begin{proof}[Proof of Theorem~\ref{infbrion}.]
Let $$\bar\Pi_l=\Pi\cap\bigcap_{i<-l}\widebar{H_i}\cap\bigcap_{i>l}\widebar{E_i}.$$ Theorem~\ref{wbrion} shows that 
\begin{equation}\label{finbrion}
G(S_{\bar\varphi}(\Pi_l))=\sum_{\substack{\bar v\text{ vertex}\\\text{of }\Pi_l}}\tau_{\bar v,l}.
\end{equation}

Obviously, the coefficients of the series in $q$ on the left stabilize onto the coefficients of $G(S_{\bar\varphi}(\bar\Pi))$. Also, for any $v$ the coefficients of the series $\tau_{\bar v,l}$ stabilize onto the coefficients of $\tau_{\bar v}$.

The remarks preceding the proof show that for any integer $N$ there is only a finite number of vertices $v$ for which $\tau_{\bar v,l}$ may contain a power of $q$ less than $N$. This shows that the infinite sum $$\sum_{\substack{\bar v\text{ vertex}\\\text{of }\Pi}}\tau_{\bar v}$$ is well-defined and that the coefficients of the right-hand side of~(\ref{finbrion}) stabilize onto this infinite sum's coefficients.
\end{proof}
%

\section{Assigning Lattice Subgraphs to Faces of $\Pi$}\label{proofintro}

First we define a subgraph $\Theta(x)\subset\mathcal R$ for any point $x\in\Pi$. The vertices of $\Theta(x)$ are all the vertices of $\mathcal R$. An edge of $\mathcal R$ connecting $(i_1,j_1)$ and $(i_2,j_2)$ is in $\Theta(x)$ if and only if $s_{i_1,j_1}(x)=s_{i_2,j_2}(x)$.

Now for a finite dimensional face $f$ we take a point $x$ such that $f$ is the minimal face containing $x$. We see that the subgraph $\Theta(x)$ does not depend on $x$ and we define $\Theta_f=\Theta(x)$. Visibly, whenever $f\subset g$ the graph $\Theta_g$ is a subgraph of $\Theta_f$. 

Relation~(\ref{shift}) shows that the graph $\Theta_f$ is invariant under the shift $(i,j)\rightarrow(i-n+1,j+n)$. This means that its connected components are divided into equivalence classes, with two components being equivalent if and only if they can be identified by an iteration of this shift. We choose a set of representatives and denote the union of these components $\Delta_f\subset \Theta_f$.

Moreover, relation~(\ref{shift}) shows that $(i,j)$ and $(i-n+1,j+n)$ are never in one component of $\Theta_f$. This means that for every integer $l$ there is exactly one vertex $(i,j)\in\Delta_f$ with $in+j(n-1)=l$. We denote this vertex $(\eta_f(l),\theta_f(l))$. We also see that the edges of $\Delta_f$ are in one-to-one correspondence with those hyperplanes $E_l$ and $H_l$ which contain $f$.

Now consider a vertex $v$ of $\Pi$. We can define a change of coordinates on $V$ in terms of the graph $\Delta_v$. The new coordinates will be labeled by pairs $(i,j)$ such that $(i,j)$ is a vertex of $\Delta_v$. The corresponding coordinate of $x$ is simply $s_{i,j}(x)$. Definition~(\ref{gtdef}) together with the previous paragraph show that this is indeed a nondegenerate change of coordinates and the new coordinates of a point are integral if and only if this point was integral.

\begin{proposition}\label{imcv}
For a point $x\in V$ we have $x\in C_v$ if and only if for any edge of $\Delta_v$ joining vertices $(i_1,j_1)$ and $(i_2,j_2)$ the coordinates $s_{i_1,j_1}(x)$ and $s_{i_2,j_2}(x)$ satisfy the corresponding inequality.
\end{proposition}
\begin{proof}
This is evident from the fact that $x\in C_v$ if and only if $x_l\ge 0$ whenever $v\in E_l$ and $\chi_l(x)\le k$ whenever $v\in H_l$.
\end{proof}

We proceed to give an extensive list of properties of the introduced objects. 

\begin{proposition}\label{vertgraph}
If $v$ is a vertex of $\Pi$, then every vertex of $\Delta_v$ is connected to one if its two upper neighbors.
\end{proposition}
\begin{proof}
This is evident from Proposition~\ref{vert}.
\end{proof}

\begin{proposition}
Whenever $(i,j)$ and $(i,j+1)$ are in the same connected component of $\Delta_v$ the vertices $(i-1,j+1)$ and $(i+1,j)$ (i.e. the two common neighbors of $(i,j)$ and $(i,j+1)$) are also in that same component of $\Delta_v$.
\end{proposition}
\begin{proof}
This evident from the fact that $s_{i,j}(v)$ is an plane-filling GT-pattern.
\end{proof}

Next visualize a cycle graph with $n$ vertices labeled $0,\ldots,n-1$ and its subgraph determined by the following rule. Vertices $i$ and $i+1$ are adjacent in the subgraph whenever $a_{i+1}=0$ (all indices are to be read modulo $n$). Since $\lambda\neq 0$ this subgraph is a disjoint union of $m(\lambda)$ path graphs of sizes $l_1,\ldots,l_{m(\lambda)}$. The numbers $m(\lambda)$ and $l_1,\ldots,l_{m(\lambda)}$ are important characteristics of $\lambda$. We point out straight away that, as is well-known, the stabilizer $$W_\lambda\simeq S_{l_1}\times\ldots\times S_{l_m}$$ and $$W_\lambda(t)=[l_1]_t!\ldots[l_m]_t!$$

\begin{proposition}\label{uplimit}
For any vertex $v$ of $\Pi$ the number of connected components in $\Delta_v$ is $m(\lambda)$. Moreover, they can be labeled $\Gamma_1,\ldots,\Gamma_{m(\lambda)}$ in such a way that for $i\ll 0$ component $\Gamma_r$ contains exactly $l_r$ vertices from row $i$.
\end{proposition}
\begin{proof}
Consider some $r\in[1,m(\lambda)]$. Due to the definition of the integers $l_r$ we can specify an integer ${I_r}$ with the following properties.
\begin{enumerate}
\item For $l<{I_r}+n^2$ one has $v_l=a_{l\bmod n}$.
\item One has $v_{I_r}=v_{{I_r}-1}=\ldots=v_{{I_r}-l_r+2}=0.$ \\($l_r-1$ consecutive terms.) 
\item One has $v_{{I_r}+1}\neq 0$ and $v_{{I_r}-l_r+1}\neq 0$.
\end{enumerate}

The above translates into the following statement about the plane-filling GT-pattern associated with $v$. Each of the elements \[s_{\eta_v({I_r}),\theta_v({I_r})}(v),\ldots,s_{\eta_v({I_r}),\theta_v({I_r})+l_r-1}(v)\] (\(l_r\) consecutive elements in row \(\eta_v({I_r})\)) is equal to its respective upper-left neighbor. That is due to Property 1 of ${I_r}$ above. Also, $s_{\eta_v({I_r}),\theta_v({I_r})}(v)$ and the $l_r-2$ elements to its right are equal to their respective upper-right neighbors. That is due to Property 2. However, $s_{\eta_v({I_r}),\theta_v({I_r})-1}(v)$ is not equal to its upper-right neighbor $s_{\eta_v({I_r})-1,\theta_v({I_r})}(v)$ and $s_{\eta_v({I_r}),\theta_v({I_r})+l_r-1}(v)$ is not equal to its upper-right neighbor $s_{\eta_v({I_r})-1,\theta_v({I_r})+l_r}(v)$. That is by Property 3.

We have established that vertex $(\eta_v({I_r}),\theta_v({I_r}))$ is in one component with the $l_r-1$ vertices to its right, as well as its upper-left neighbor and the $l_r-1$ vertices to that neighbor's right. We have also seen that this component has no other vertices in rows $\eta_v({I_r})$ and $\eta_v({I_r})-1$.

We see that there are indeed $m(\lambda)$ components $\Gamma_1,\ldots,\Gamma_{m(\lambda)}$ such that for $$i\le \min_r(\eta_v(I_r))$$ component $\Gamma_r$ contains exactly $l_r$ vertices from row $i$. It remains to observe that for all $l\le\min_r(I_r)$ the vertex $(\eta_v(l),\theta_v(l))$ is contained in one of those $m(\lambda)$ components. This shows that there are no other components since Proposition~\ref{vertgraph} implies that every component of $\Delta_v$ has vertices in row $i$ for $i\ll 0$.
\end{proof}

\begin{proposition}\label{upsame}
For a point $x\in C_v$ one has $s_{i,j}(x)=s_{i,j}(v)$ when $(i,j)\in\Delta_v$ and $i\ll 0$.
\end{proposition}
\begin{proof}
Obviously, there exists an integer $M$ such that $s_{\eta_v(l),\theta_v(l)}(x)=s_{\eta_v(l),\theta_v(l)}(v)$ whenever $l<M$. However, Proposition~\ref{uplimit} shows that for $i\ll 0$ we have $l<M$ whenever $\eta_v(l)=i$.
\end{proof}

We have another proposition describing $\Delta_v$ in rows $i\gg 0$.
\begin{proposition}\label{downlimit}
Only one of the $m(\lambda)$ components of $\Delta_v$ contains vertices $(i,j)$ with arbitrarily large $i$. For $i\gg 0$ this component contains a single vertex in row $i$.
\end{proposition}
\begin{proof}
For $l\gg 0$ we have $v_l=0$ which shows that $$s_{\eta_v(l),\theta_v(l)}(v)=s_{\eta_v(l)-1,\theta_v(l)+1}(v).$$ Consequently, for any $l>0$ we have $$(\eta_v(l),\theta_v(l))=(\eta_v(l-1)+1,\theta_v(l-1)-1)$$ and the two vertices are adjacent in $\Delta_v$. Since this holds for all $l\gg 0$, the proposition is proved.
\end{proof}

\begin{proposition}\label{downsame}
For a point $x\in C_v$ all the coordinates $s_{i,j}(x)$ with $(i,j)\in\Delta_v$ and $i\gg 0$ are the same.
\end{proposition}
\begin{proof}
For $l\gg 0$ we have $x_l=0$ which entails $s_{\eta_v(l),\theta_v(l)}(x)=s_{\eta_v(l-1),\theta_v(l-1)}(x).$ We then apply Proposition~\ref{downlimit}.
\end{proof}

For a point $x\in C_v$ how do we express the monomial $G(e^{\bar x})$ via the coordinates $s_{i,j}(x)$? This question is best answered in terms of the array $$s_{i,j}(x,v)=s_{i,j}(x)-s_{i,j}(v).$$

\begin{proposition}\label{zpow}
For an integer point $x\in C_v$  the power in which $G(e^{\bar x-\bar v})$ contains $z_r$ is equal to $$\sum\limits_{i\equiv r\bmod(n-1)} \left(\sum\limits_{(i,j)\in\Delta_v}s_{i,j}(x,v)-\sum\limits_{(i-1,j)\in\Delta_v}s_{i-1,j}(x,v)\right).$$
\end{proposition}
\begin{proof}
Formula~(\ref{zweight}) shows that $G(e^{\bar x-\bar v})$ contains $z_r$ in the power
$$\sum\limits_{l\equiv r\bmod(n-1)}(x_l-v_l)=\\\sum\limits_{l\equiv r\bmod(n-1)} (s_{\eta_v(l),\theta_v(l)}(x,v)-s_{\eta_v(l-1),\theta_v(l-1)}(x,v)).$$ Now it remains to apply $$l=n\eta_v(l)+(n-1)\theta_v(l)\equiv\eta_v(l)\bmod(n-1).$$

Propositions~\ref{upsame} and~\ref{downsame} show that all the sums in consideration have a finite number of nonzero summands.
\end{proof}

\begin{proposition}\label{qpow}
For an integer point $x\in C_v$  we have $$\deg G(e^{\bar x-\bar v})=\sum\limits_{i\equiv 0\bmod(n-1)} \sum\limits_{(i,j)\in\Delta_v}(-s_{i,j}(x,v)+S_{i,j}),$$ where $S_{i,j}=0$ when $in+j(n-1)<0$ and $S_{i,j}=\sum_l(x_l-v_l)$ if $in+j(n-1)\ge 0$.
\end{proposition}
\begin{proof}
Via~(\ref{qweight}) we have
\begin{multline*}
\deg G(e^{\bar x-\bar v})=-\sum\limits_{r<0}\sum\limits_{l\le r(n-1)}(x_l-v_l)+\sum\limits_{r\ge0}\left(\sum_{l=-\infty}^\infty(x_l-v_l)-\sum\limits_{l\le r(n-1)}(x_l-v_l)\right)=\\\sum\limits_{r\in\mathbb{Z}}\left(S_{\eta_v(r(n-1)),\theta_v(r(n-1))}-s_{\eta_v(r(n-1)),\theta_v(r(n-1))}(x,v)\right).
\end{multline*}
We then apply $\eta_v(r(n-1))\equiv 0\bmod(n-1)$. Note that we again have a finite number of nonzero summands in every sum.
\end{proof}
%
%

Further, the weight $\varphi(f)$ has a nice interpretation in terms of the graph $\Delta_f$.
\begin{proposition}\label{phigraph}
For a face $f$ and integer $l>0$ let $d_l$ be the number of pairs $(\Gamma,i)$ with $\Gamma$ a connected component of $\Delta_f$ and $i$ an integer such that $\Gamma$ has $l$ vertices in row $i$ and $l-1$ vertices in row $i-1$. Then $\varphi(f)=\prod(1-t^l)^{d_l}$.
\end{proposition}
\begin{proof}
Straightforward from the definitions.
\end{proof}

\begin{proposition}\label{dimf}
For a finite-dimensional face $f$ we have $$\dim f=|\{\text{components of }\Delta_f\}|-m(\lambda).$$
\end{proposition}
\begin{proof}
If $f$ is a vertex this follows from Proposition~\ref{uplimit}. If $f$ is not a vertex it has a nonempty interior with the same dimension.

Choose a point $x\in V$ from the interior of $f$. For any two vertices $(i_1,j_1)$ and $(i_2,j_2)$ of $\Delta_f$ that are adjacent in $\mathcal R$ we have $s_{i_1,j_1}(x)=s_{i_2,j_2}(x)$ if and only if the two vertices are adjacent in $\Delta_f$.

Consider a vertex $v$ of $f$. Since $\Theta_f\subset\Theta_v$, we may assume that $\Delta_f\subset\Delta_v$. Proposition~\ref{upsame} together with Proposition~\ref{uplimit} then shows that there are $m(\lambda)$ components of $\Delta_f$ that meet arbitrarily high rows. If $(i,j)$ is a vertex in one of these components, then $s_{i,j}(x)=s_{i,j}(v)$. Thus, from the previous paragraph we see that we have exactly $|\{\text{components of }\Delta_f\}|-m(\lambda)$ degrees of freedom when choosing the coordinates $s_{i,j}(x)$ (with respect to vertex $v$).
\end{proof}

We finish this section off by showing, as promised, that $G(e^{\varepsilon_{v,l}})\neq 1$.

Consider a vertex $v$ and an edge $e$ containing $v$. We may assume that $\Delta_e$ is a subgraph of $\Delta_v$. According to Proposition~\ref{dimf} the graph $\Delta_e$ has $m(\lambda)+1$ connected components of which only one does not meet arbitrarily high rows. Denote that component $\Gamma_e\subset\Delta_e$. Let $\varepsilon_{v,l}$ be the generator of $e$.

\begin{proposition}\label{edges}
In the above notations the array $$s_{i,j}(\varepsilon_{v,l})=s_{i,j}(v+\varepsilon_{v,l})-s_{i,j}(v)$$ with $(i,j)$ ranging over the vertices of $\Delta_e$ has the following description. If $(i,j)$ is outside of $\Gamma_e$, then $s_{i,j}(\varepsilon_{v,l})=0$. For all $(i,j)$ in $\Gamma_e$ the value $s_{i,j}(\varepsilon_{v,l})$ is the same and equal to either $-1$ or 1.
\end{proposition}
\begin{proof}
Proposition~\ref{upsame} shows that for any point $x$ of $e$ for $(i,j)$ outside of $\Gamma_e$ we indeed have $s_{i,j}(x)=s_{i,j}(v)$. Moreover, by definition for $x\in e$ all of its coordinates $s_{i,j}(x)$ with $(i,j)$ within $\Gamma_e$ must be the same. By taking $x=v$ and $x=v+\varepsilon_{v,l}$ we obtain the Proposition.
\end{proof}

\begin{proposition}
For any vertex $v$ and generator $\varepsilon_{v,l}$ we have $G(e^{\varepsilon_{v,l}})\neq 1$.
\end{proposition}
\begin{proof}
$$G(e^{\varepsilon_{v,l}})=G(e^{(\bar v+\varepsilon_{v,l})-\bar v}).$$ This monomial may be calculated via Propositions~\ref{zpow} and~\ref{qpow}. More specifically, we see that there are three possible cases.
\begin{enumerate}
\item $\Gamma_e$ is finite and does not intersect any row $i$ with $i\equiv 0\bmod (n-1)$. From Proposition~\ref{zpow} we then we see that $z_{i_0\bmod(n-1)}$ occurs in a nonzero power, where $i_0$ is the highest row containing vertices from $\Gamma_e$.
\item $\Gamma_e$ is finite and intersects some row with $i\equiv 0\bmod (n-1)$. Then $\deg G(e^{\varepsilon_{v,l}})\neq 0$ since for vertices $(i,j)$ of $\Delta_e$ with $i\gg 0$ we have $s_{i,j}(\varepsilon_{v,l})=0$ and thus all the values $S_{i,j}$ from Proposition~\ref{qpow} are zero.
\item $\Gamma_e$ is infinite. This means that for $i\gg 0$ there is a single vertex of $\Gamma_e$ in row $i$. Proposition~\ref{zpow} then shows that the sum of powers in which the $z_r$ occur is 1.
\end{enumerate}
\end{proof}

\section{Proof of Theorem~\ref{contrib}}\label{last}

In this Section we finally apply the tools developed in Part~\ref{tools} combining them with the Propositions from the previous section.

The vertex $v$ of $\Pi$ is fixed throughout this section. Denote $\Gamma_1,\ldots,\Gamma_{m(\lambda)}$ the connected components of $\Delta_v$. For $(i,j)\in\Gamma_r$ all the numbers $s_{i,j}(v)$ are the same, let them be equal to $b_r$.

Choose an integer $M_v$ such that for $i\ge M_v$ row $i$ meets $\Delta_v$ in exactly one vertex while for $i\le -M_v$ row $i$ meets component $\Gamma_r$ of $\Delta_v$ in $l_r$ vertices and, furthermore, one has $in+j(n-1)<0$ for any vertex $(i,j)$ of $\Delta_v$ in row $-M_v$ or above. Propositions~\ref{uplimit} and~\ref{downlimit} show that such a $M_v$ exists.

For $l\ge M_v$ denote $D_l$ the section of $C_v$ comprised of points $x\in C_v$ with the following properties.
\begin{enumerate}
\item For $i\le -l$ we have $s_{i,j}(x)=s_{i,j}(v)$ for all vertices $(i,j)$ of $\Delta_v$.
\item For $i\ge l$ all the coordinates $s_{i,j}(x)$ with $(i,j)$ a vertex of $\Delta_v$ are the same.
\end{enumerate} 
An important observation is that $D_l$ is a finite dimensional face of $C_v$. Indeed, $D_l$ is defined as the intersection of $C_v$ and all the hyperplanes $E_i\ni v$ and $H_i\ni v$ except for a finite number. 

Now, the rational function $G(\sigma_{\bar\varphi}(\widebar{D_l}))$ may be viewed as an element of $\mathfrak S$ which we denote $\sigma_l$.
\begin{lemma}\label{limit}
The series $\sigma_l$ converge coefficient-wise to the series $\tau_{\bar v}$.
\end{lemma}
\begin{proof}
We consider $G(\sigma_{\bar\varphi}(\widebar{D_l}))$ to be a fraction the denominator of which is the product of $1-G(e^\varepsilon)$ over all generators $\varepsilon$ of edges of $D_l$. The coefficients of these denominators visibly converge to the coefficients of $$(1-G(e^{\varepsilon_{v,1}})(1-G(e^{\varepsilon_{v,1}})\ldots$$ We are thus left to prove that the numerators converge coefficient-wise to $Q_v$.

This is done in complete analogy with the argument proving Lemma~\ref{welldef}. The only difference is that in the last paragraph we use the characterization of the generators given by Proposition~\ref{edges} rather than the one taken from~\cite{me2}.
\end{proof}

On the other hand, let $\Delta_{v,l}$ be the full subgraph of $\Delta_v$ obtained by removing all rows with number less than $-l$ or greater than $l$. Such a $\Delta_{v,l}$ has $m(\lambda)$ connected components each of which is an ordinary graph. We denote these components $\Gamma^l_r\subset \Gamma_r$.

We can now see that we have a natural bijection $$\xi_l:D_{\Gamma^l_1}(b_1,\ldots,b_1)\times\ldots\times D_{\Gamma^l_{m(\lambda)}}(b_{m(\lambda)},\ldots,b_{m(\lambda)})\rightarrow D_l.$$ (Recall that every factor on the right is a cone with vertex $v_{\Gamma_r^l}(b_r)$.) The coordinate
\begin{equation}\label{imcoord}
s_{i,j}(\xi_l(x_1\times\ldots\times x_{m(\lambda)}))
\end{equation}
is equal to the corresponding coordinate of $x_r$ when $(i,j)\in\Gamma_r^l$. When $(i,j)$ is a vertex of $\Delta_v$ with $i<-l$ the coordinate~(\ref{imcoord}) is equal to $s_{i,j}(v)$ and for $i\ge l$ those coordinates are all the same. Proposition~\ref{imcv} shows that this is indeed a bijection. We also have the corresponding bijection $\widebar{\xi_l}$ with image $\widebar{D_l}$.

Propositions~\ref{zpow} and~\ref{qpow} together with our choice of $l$ show that for a certain specialization $\Psi_l$ substituting each $x_i$ with a monomial in $z_1,\ldots,z_{n-1},q$ the following holds. For any tuple of integer points $x_r\in D_{\Gamma^l_r}(b_r,\ldots,b_r)$ we have
\begin{equation}\label{expprod}
G\left(e^{\widebar{\xi_l}(x_1\times\ldots\times x_{m(\lambda)})}\right)=G(e^{\bar v})\Psi_l\left(\prod_{r=1}^{m(\lambda)}F\left(e^{\left(x_r-v_{\Gamma_r^l}(b_r)\right)}\right)\right).
\end{equation}
It is straightforward to describe $\Psi_l$ explicitly, we, however, will not make use of such a description and therefore omit it.

Proposition~\ref{phigraph} together with $l\ge M_v$ shows that for a face of $D_l$ $$f=\xi_l(f_1\times\ldots\times f_{m(\lambda)})$$ with $f_r$ being a face of $D_{\Gamma_r^l}(b_r,\ldots,b_r)$ we have the following identity.
\begin{equation}\label{weightprod}
\varphi(f)=\prod_{r=1}^{m(\lambda)}\varphi_{\Gamma_r^l}(b_r,\ldots,b_r)(f_r).
\end{equation}

Combining~(\ref{expprod}) and~(\ref{weightprod}) we, finally, obtain
\begin{equation}\label{psiprod}
G\left(\sigma_{\bar\varphi}\left(\widebar{D_l}\right)\right)=G(e^{\bar v})\Psi_l\left(\prod_{r=1}^{m(\lambda)}F\left(e^{-v_{\Gamma_r^l(b_r)}}\right)\psi_{\Gamma_r^l}(b_r,\ldots,b_r)\right).
\end{equation}

Now it is time to define the distinguished set of vertices from Theorem~\ref{contrib} which we again refer to as ``relevant". A vertex $v$ is not relevant if and only if the graph $\Delta_v$ has a connected component $E$ with the following property. $E$ has more vertices in row $i+1$ than in row $i$ for some integer $i$.

This definition together with $l\ge M_v$ immediately implies that if $v$ is non-relevant, then one of the components $\Gamma_r^l$ of $\Delta_{v,l}$ contains more vertices in some row than the row above. Combining~(\ref{psiprod}) with Theorem~\ref{zero} and then employing Lemma~\ref{limit} now proves part b) of Theorem~\ref{contrib}.

We move on to considering a relevant $v$. We first discuss the case of a regular $\lambda$, i.e. all $a_i$ being positive. In this case $\Delta_v$ has $n$ components each of which contains a single vertex in row $i$ for $i\ll 0$ (Proposition~\ref{uplimit}).
\begin{proposition}\label{relvert}
For a regular $\lambda$ vertex $v$ of $\Pi$ is non-relevant if and only if we have an $l$ for which $v_l=0$ and $v_{l+n-1}\neq 0$.
\end{proposition}
\begin{proof}
If $v$ is non-relevant we have a component of $\Delta_v$ which contains one vertex $(i-1,j)$ in row $i-1$ and two vertices $(i,j-1)$ and $(i,j)$ in row $i$. This, in particular, shows that $v_{in+(j-1)(n-1)}=0$ while $v_{in+j(n-1)}\neq 0$ which proves the ``only if" part.

Conversely, if $v_l=0$ and $v_{l+n-1}\neq 0$, then in $\Delta_v$ the vertex $(\eta_v(l),\theta_v(l))$ is connected to its upper-right neighbor, while $(\eta_v(l+n-1),\theta_v(l+n-1))$ is not. This, however, means that $(\eta_v(l+n-1),\theta_v(l+n-1))$ is connected to its upper-left neighbor (Proposition~\ref{vertgraph}). This upper-left neighbor is then $(\eta_v(l-1),\theta_v(l-1))$ who is also the upper-right neighbor of $(\eta_v(l),\theta_v(l))$. We therefore see that the corresponding component contains two vertices in row $\eta_v(l)$ which leads to $v$ being non-relevant.
\end{proof}

Such an interpretation of relevant vertices for the case of regular $\lambda$ is in accordance with the one found in~\cite{fjlmm2}. The following information can then be extracted from papers~\cite{fjlmm2} and~\cite{me2}.
\begin{proposition}\label{simpvert}
\begin{enumerate}[label=\alph*)]
\item For regular $\lambda$ the relevant vertices are enumerated by elements of the Weyl group $W$. If $v_w$ is the vertex corresponding to $w\in W$, then $e^{\mu_{v_w}}=w\lambda$.
\item All the cones $D_l$ are simplicial and unimodular.
\item The multiset $\{G(e^{\varepsilon_{v_w,i}})\}$ coincides with the multiset $\{e^{-w\alpha},\alpha\in\Phi^+\}$, where each $\alpha$ is counted $m_\alpha$ times.
\end{enumerate}
\end{proposition}

All we, essentially, are left to prove is the following.
\begin{proposition}
Let $\lambda$ be regular and $v$ be a relevant vertex. For any face $f$ of cone $D_l$ we have $$\varphi(f)=(1-t)^{\dim f}.$$
\end{proposition}
\begin{proof}
We may assume that $\Delta_f$ is a subgraph of $\Delta_v$.

All $n$ connected components of $\Delta_v$ are infinite path graphs, $n-1$ of them infinite in one direction (``up") and one infinite in both directions. This together with Proposition~\ref{phigraph} shows that $\varphi(f)=(1-t)^d$, where $d$ is the number of vertices in $\Delta_f$ not adjacent to any vertex in the row above. 

However, Proposition~\ref{dimf} shows that $\dim f=d$ as well.
\end{proof}

The above Proposition together with part b) of Proposition~\ref{simpvert} shows that $G(\sigma_{\bar\varphi}(\widebar{D_l}))$ is the product of $F(e^{\bar v})=e^{\mu_v-\lambda}$ and the quotients $$\frac{1-tF(e^\varepsilon)}{1-F(e^\varepsilon)}$$ over all generators $\varepsilon$ of edges of $D_l$. Applying parts a) and c) of Proposition~\ref{simpvert} and then Lemma~\ref{limit} now proves part a) of Theorem~\ref{contrib} in the case of regular $\lambda$.

On to the case of $\lambda$ being singular, i.e. having at least one $a_i=0$. This case will be deduced from the regular case, so we introduce $\lambda^1$, an arbitrary integral dominant regular weight. We denote the objects corresponding to $\lambda^1$ by adding a $^1$ superscript, e.g. $\Pi^1$, $\varphi^1$, $E_l^1$ etc.

Due to Proposition~\ref{relvert} the relevant vertices of $\Pi^1$ are parametrized by sequences $y=(y_i)$ infinite in both directions with $y_i\in\{0,1\}$ and having the following properties.
\begin{enumerate}
\item For $l\gg 0$ one has $y_l=0$.
\item For $l\ll 0$ one has $y_l=1$.
\item One has $y_{l+n-1}=0$ whenever $y_l=0$.
\end{enumerate}
The vertex $v_y^1$ corresponding to such a sequence is uniquely defined by $v_y^1\in E_l^1$ whenever $y_l=0$ and $v_y^1\in H_l^1$ whenever $y_l=1$. The fact that the $a_i^1$ are all positive implies that different $y$ define different $v_y^1$. Since the relevant vertices of $\Pi^1$ are also parametrized by the affine Weyl group $W$ for each $y$ we may define $w_y\in W$ such that $v_y^1=v_{w_y}^1$. Clearly, $w_y$ does not depend on $\lambda_1$ but only on $y$.

Each sequence $y$ also defines a vertex $v_y$ of $\Pi$ by the same rule. However, some of these $v_y$ may coincide. 
\begin{proposition}
The vertices $v_y$ are precisely the relevant vertices of $\Pi$. For any $y$ we have $\mu_{v_y}=w_y\lambda$.
\end{proposition}
\begin{proof}
For any $v_y^1$ we see that if a hyperplane $H_l^1$ or $E_l^1$ contains $v_y^1$, then the corresponding hyperplane $H_l$ or $E_l$ must contain $v_y$. This means that we may assume that the graph $\Delta_{v_y^1}$ is a subgraph of $\Delta_{v_y}$ (with the same set of vertices).

However, visibly, if a component $E$ of $\Delta_{v_y}$ contained more vertices in some row $i$ than in row $i-1$, then so would one of the components of $\Delta_{v^1_y}$ contained in $E$. This would contradict $v^1_y$ being relevant.

Conversely, since every component of $\Delta_{v^1_y}$ contains no less vertices in any row than in the row below, the same holds for every component of $\Delta_{v_y}$. That is because every component of $\Delta_{v_y}$ is obtained by joining components of $\Delta_{v^1_y}$.

For the second part, note that whether $\lambda$ is regular or not, the point $v_y$ depends linearly on $\lambda$ and $\mu_{v_y}$ depends linearly on $v_y$. Thus $\mu_{v_y}$ depends linearly on $\lambda$.
\end{proof}

We now see that the relevant vertices of $\Pi$ do indeed correspond to elements of the orbit $W\lambda$. 

To prove part a) of Theorem~\ref{contrib} for singular $\lambda$ it now suffices to show that 
\begin{equation}\label{decomp1}
\tau_{\bar v}=\frac 1{[l_1]_t!\ldots[l_{m(\lambda)}]_t!}\sum_{v_y=v}G\left(e^{\bar v-\bar v_y^1}\right)\tau_{\bar v_y^1}.
\end{equation}

For a vertex $v_y^1$ with $v_y=v$ and integer $l\ge M_{v_y^1}$ denote $D_{y,l}^1$ the corresponding face of $C_{v_y^1}$. Choose an $l$ greater than $M_v$ and all of the $M_{v_y^1}$. Due to Lemma~\ref{limit}, identity~(\ref{decomp1}) will follow from
\begin{equation}\label{decomp2}
G\left(\sigma_{\bar\varphi}\left(\widebar{D_l}\right)\right)=\frac 1{[l_1]_t!\ldots[l_{m(\lambda)}]_t!}\sum_{v_y=v} G\left(e^{\bar v-\bar v_y^1}\right)G\left(\sigma_{\bar\varphi^1}\left(\widebar{D_{y,l}^1}\right)\right).
\end{equation}

For all $v_y^1$ with $v_y=v$ the coordinate $s_{i,j}^1(v_y^1)$ is the same when $(i,j)\in\Delta_v$ and $i\le-l$. Denote the sequence of numbers $s_{-l,j}^1(v_y^1)$ with $(-l,j)\in\Gamma_r$ via $c_1^r,\ldots,c_{l_r}^r$. Also, for any such $v_y^1$ the coordinates $s_{i,j}^1(v_y^1)$ with $(i,j)\in\Delta_v$ and $i\ge l$ are all the same. Now consider the polyhedron $D_l^1\subset V^1$ consisting of such $x^1$ that
\begin{enumerate}
\item For any $i\le -l$ the $l_r$ coordinates $s_{i,j}^1(x^1)$ with $(i,j)\in\Gamma_r$ are equal to $c_1^r,\ldots,c_{l_r}^r$ from left to right.
\item All the coordinates $s_{i,j}^1(x^1)$ in rows $i\ge l$ are the same.
\item The coordinates $s_{i,j}^1(x^1)$ satisfy all the inequalities corresponding to edges of $\Delta_v$.
\end{enumerate}

Any vertex of $D_l^1$ is a vertex of $\Pi^1$ and the faces of $D_l^1$ correspond naturally to faces of $\Pi^1$ which allows one to define $\varphi_1$ in faces of  $D_l^1$. The vertices of $D_l^1$ that are relevant vertices of $\Pi^1$ are precisely the $v_y^1$ with $v_y=v$. The weighted Brion Theorem for $D_l^1$ (after application of $G$) reads
$$
G\left(\sigma_{\bar\varphi^1}\left(\widebar{D_l^1}\right)\right)=\sum_{v_y} G\left(\sigma_{\bar\varphi^1}\left(\widebar{D_{l,y}^1}\right)\right).
$$
We know that the contributions of other vertices are zero, having already discussed non-relevant vertices.

Clearly, $D_l$ is a degeneration of $D_l^1$, let $\pi$ be the corresponding map between face sets. With previous paragraph taken into account, Lemma~\ref{wdegen} provides
$$
G\left(\sigma_{\bar\varphi'}\left(\widebar{D_l}\right)\right)=\sum_{v_y=v} G\left(e^{\bar v-\bar v_y^1}\right)G\left(\sigma_{\bar\varphi^1}\left(\widebar{D_{y,l}^1}\right)\right),
$$
where for a face $f$ of $D_l$
$$
\varphi'(f)=\sum_{g\in\pi^{-1}(f)}(-1)^{\dim g-\dim f}\varphi^1(g).
$$
All that remains to be shown is that for any $f$ we have
\begin{equation}\label{phirel}
\varphi'(f)=[l_1]_t!\ldots[l_{m(\lambda)}]_t!\varphi(f).
\end{equation}

Now, visibly, we have a bijection 
$$
\xi^1_l:D_{\Gamma_1^l}\left(c_1^1,\ldots,c_{l_1}^1\right)\times\ldots\times D_{\Gamma_{m\left(\lambda\right)}^l}\left(c_1^{m\left(\lambda\right)},\ldots,c_{l_{m\left(\lambda\right)}}^{m\left(\lambda\right)}\right)\rightarrow D_l^1.
$$ 
Moreover, for a face $g$ of $D_l^1$ we have $$\varphi^1(g)=\prod_{r=1}^{m(\lambda)}\varphi_{\Gamma_r^l}(c_1^r,\ldots,c_{l_r}^r)(g_r),$$
where $g=\xi_l^1(g_1\times\ldots\times g_{m(\lambda)})$.

Recall that $D_{\Gamma_r^l}(b_r,\ldots,b_r)$ is a degeneration of $D_{\Gamma_r^l}(c_1^r,\ldots,c_{l_r}^r)$, let $\pi_r$ be the corresponding map between face sets. Visibly, for $g=\xi_l^1(g_1\times\ldots\times g_{m(\lambda)})$ a face of $D_l^1$ we have $$\pi(g)=\xi_l(\pi_1(g_1)\times\ldots\times\pi_{m(\lambda)}(g_{m(\lambda)})).$$ This shows that~(\ref{phirel}) for $f=\xi_l(f_1\times\ldots\times f_{m(\lambda)})$ may be obtained by multiplying together the identities provided by Lemma~\ref{graphsum} for degenerations $\pi_r$ and faces $f_r$ respectively.

We have proved Theorem~\ref{contrib} and, via Theorem~\ref{infbrion}, the main Theorem~\ref{main} follows.

\addtocontents{toc}{\vspace{4pt}}

Boris Feigin: National Research University Higher School of Economics, International Laboratory of Representation Theory and Mathematical Physics, Myasnitskaya
ulitsa 20, 101000 Moscow, Russia and Landau Institute for Theoretical Physics, prospekt Akademika Semenova 1A, 142432 Chernogolovka, Russia. 

Email: borfeigin@gmail.com.
\\

Igor Makhlin: National Research University Higher School of Economics, International Laboratory of Representation Theory and Mathematical Physics, Myasnitskaya ulitsa 20, 101000 Moscow, Russia and Landau Institute for Theoretical Physics, prospekt Akademika Semenova 1A, 142432 Chernogolovka, Russia. Supported in part by the Simons Foundation and the Moebius Contest Foundation for Young Scientists. 

Email: imakhlin@mail.ru.

\end{document}